\documentclass[a4paper,11pt]{amsart}

\usepackage{amsmath,amssymb,mathrsfs,amsthm}
\usepackage{bm,esint,mathtools,xcolor}
\usepackage[utf8]{inputenc}
\usepackage[left=3.2cm,right=3.2cm,top=3cm,bottom=3cm]{geometry}
\usepackage{hyperref,hypcap}
\hypersetup{colorlinks=true,citecolor=purple,linkcolor=blue,bookmarksdepth=3}
\usepackage[inline]{enumitem}
\usepackage{microtype}

\usepackage[backend=bibtex,style=numeric-comp,sorting=nyt,giveninits=true,doi=false,isbn=false,url=false,maxbibnames=99,maxcitenames=99]{biblatex}
\AtEveryBibitem{\clearlist{language}}
\usepackage{caption}
\theoremstyle{plain}
\newtheorem{theorem}{Theorem}[section]
\newtheorem{lemma}[theorem]{Lemma}
\newtheorem{proposition}[theorem]{Proposition}
\newtheorem{corollary}[theorem]{Corollary}
\newtheorem{conjecture}[theorem]{Conjecture}
\theoremstyle{definition}
\newtheorem{definition}[theorem]{Definition}
\theoremstyle{remark}
\newtheorem{remark}[theorem]{Remark}

\numberwithin{equation}{section}
\numberwithin{figure}{section}
\numberwithin{table}{section}
\newcommand{\diff}{\mathop{}\!\mathrm{d}}
\newcommand{\dist}{\mathrm{dist}}

\newcommand{\torus}{\mathbb{T}}
\newcommand{\restrict}{{\restriction}}
\newcommand{\lattice}{\mathbb{L}}
\newcommand{\affine}{\mathbb{A}}

\newcommand{\hamiltonian}{\mathcal{H}}
\newcommand{\schrodinger}{\mathcal{U}}
\newcommand{\galilean}{\mathcal{G}}
\newcommand{\trig}{\mathscr{T}}

\newcommand{\distribution}{\mathscr{D}}

\DeclareMathOperator{\diam}{diam}

\DeclareMathOperator{\proj}{proj}
\DeclareMathOperator{\lsp}{span}
\DeclareMathOperator{\supp}{supp}

\allowdisplaybreaks

\begin{document}

\title[Observability of Schrödinger Propagators]{Observability of Schrödinger Propagators \\
on Tori in Rough Settings}

\author[N.~Burq]{Nicolas Burq}
\address[Nicolas Burq]{Laboratoire de Mathématiques d'Orsay, Université Paris-Saclay}
\email{nicolas.burq@universite-paris-saclay.fr}

\author[H.~Zhu]{Hui Zhu}
\address[Hui Zhu]{New York University Abu Dhabi, Division of Science}
\email{hui.zhu@nyu.edu}

\begin{abstract}
    On tori of arbitrary dimensions, Schrödinger propagators with bounded potentials are conjectured to be observable from space-time domains of positive Lebesgue measure. 
    We reduce this conjecture to certain integrability bounds for free Schrödinger waves, thereby proving the conjecture on the one-dimensional torus and producing new examples of observation domains.
    These bounds are far weaker than Bourgain's conjectured periodic Strichartz estimates, yet remain highly nontrivial.
\end{abstract}

\maketitle


\section{Introduction}

On the $d$-dimensional torus $\torus^d = \mathbb{R}^d / (2\pi \mathbb{Z})^d$, let $V$ be a time-independent real potential and define the Schrödinger operator $\hamiltonian_V = -\Delta + V$, which is assumed to be essentially self-adjoint with domain $C^\infty(\torus^d)$.
Our principal objective is to investigate the observability of the Schrödinger propagator
$\schrodinger_V(t) = e^{-it \hamiltonian_V}$.

\subsection{Observability and Applications}

Recall that, the Schrödinger propagator $\schrodinger_V$ is \emph{observable} from a Lebesgue measurable domain $\Omega \subset \mathbb{R} \times \torus^d$ if for some constant $C_{d,V,\Omega} > 0$ and for all $u_0 \in L^2(\torus^d)$, there holds
\begin{equation}
    \label{eq::observability}
    \|u_0\|_{L^2(\torus^d)}
    \le C_{d,V,\Omega} \|\schrodinger_V u_0\|_{L^2(\Omega)}.
\end{equation}
When $V \in L^\infty(\torus^d)$ and $\Omega$ is nonempty and open, the observability is conjectured to hold \cite{BurqZworski2012-2d-smooth,BourgainBurqZworski2013rough-potential,Anantharaman-Macia2014measure}.
We extend this conjecture to domains with positive Lebesgue measure.

\begin{conjecture}
    \label{conj::L-infty-observability}
    On a torus of arbitrary dimension, Schrödinger propagators with bounded potentials are observable from space-time domains of positive measure.
\end{conjecture}

Our main result is a criterion for observability (Theorem~\ref{thm::main-general}), which we view as a significant step toward fully resolving this conjecture.
Before stating this criterion in its full generality, we give one of its consequences in the following theorem.

\begin{theorem}
    \label{thm::main-cartesian}
    Let $A$ be the \emph{closed subalgebra} of $L^\infty(\mathbb{R} \times \torus^d)$ (i.e., a closed subspace stable under multiplication) generated by all elements that are either
    \begin{enumerate}[label=(\roman*)]
        \item locally Riemann integrable (in particular, continuous), or
        \item dependent on at most two of the space-time variables $t,x_1,\ldots,x_d$.
    \end{enumerate}
    The Schrödinger propagator $\schrodinger_V$ is observable from $\Omega$ if
    \begin{equation}
    \label{eq::cdt-V-Omega-A}
        V \in A,
        \quad
        \bm{1}_\Omega \in A \setminus \{0\}.
    \end{equation}
    In particular, Conjecture~\ref{conj::L-infty-observability} holds on the one-dimensional torus.
\end{theorem}

Throughout the paper, including in the assumption \eqref{eq::cdt-V-Omega-A}, we regard spatial functions such as $V$ as time-independent space-time functions.

An example of $\Omega$ satisfying $\bm{1}_\Omega \in A \setminus \{ 0 \}$ is a \emph{Cartesian product} of subsets that are either open or at most two-dimensional.
Precisely, write $\mathbb{R} \times \torus^d = \prod_{\ell = 0}^m E_\ell$ with $E_\ell$ being products of factors among $\mathbb{R}_t$ and $\torus_{x_j}$ ($1 \le j \le d$) and choose $\omega_\ell \subset E_\ell$ with $|\omega_\ell| > 0$.
Assume that in addition $\omega_0$ is open and $\dim E_\ell \le 2$ for $\ell \ge 1$.
Then one may take
\begin{equation*}
    \Omega = \prod_{0 \le \ell \le m} \omega_\ell.
\end{equation*}

\subsubsection{Exact Controllability}

An immediate application of Theorem~\ref{thm::main-cartesian} arises in the control theory.
Recall that, the propagator $\mathcal{U}_V$ is called \emph{exactly controllable} from $\Omega$ within time $T>0$ if for all initial and final states $u_0,u_1 \in L^2(\torus^d)$, there exists an exterior force $F \in L^2(\mathbb{R} \times \torus^d)$ which is supported in $\overline{\Omega}$, such that it sends $u\restrict_{t=0} = u_0$ to $u\restrict_{t=T} = u_1$ via the forced (i.e., inhomogeneous) Schrödinger equation
\begin{equation*}
    i\partial_t u = \hamiltonian_V u + F.
\end{equation*}

By the Hilbert uniqueness method (Lions \cite{Lions1988control-v1,Lions1988control-v2,Lions1988control-english}), the exact controllability of $\mathcal{U}_V$ from $\Omega$ within time $T$ is equivalent to its observability from $\Omega \cap ([0,T] \times \torus^d)$.
As a consequence of Theorem~\ref{thm::main-cartesian}, one has the following corollary.

\begin{corollary}
\label{cor::control-examples}
    Let $T>0$. The following controllability results hold:
    \begin{enumerate}
        \item If $d = 1$, then Schrödinger propagators of bounded potentials are exactly controllable from $\Omega$ within time $T$ provided that $|\Omega \cap ([0,T] \times \torus^d)| > 0$.
        \item If $d = 2$, then Schrödinger propagators of bounded potentials are exactly controllable from $\Omega$ within time $T$ provided that $\Omega = J \times \omega$ where $J \subset \mathbb{R}$ and $\omega \subset \torus^2$ are measurable subsets satisfying $|J \cap [0,T]| > 0$ and $|\omega| > 0$.
        \item If $d \ge 3$ and $V \in A$, then the Schrödinger propagator $\schrodinger_V$ is exactly controllable from $\Omega$ within time $T$ provided that $\Omega = J \times \omega$ where $J \subset \mathbb{R}$ is a measurable subset satisfying $|J \cap [0,T]| > 0$ and $\omega \subset \torus^2$ is a nonempty open subset.
    \end{enumerate}
\end{corollary}

These statements, while not fully exploiting the strength of our criterion, provide the first controllability results for Schrödinger propagators in general dimensions from space-time measurable domains.
We will extend it to broader geometric settings in a forthcoming work \cite{BZ-control-temporal}.
Earlier studies in this direction either considered only Cartesian products of time intervals with measurable spatial sets \cite{BourgainBurqZworski2013rough-potential}, or focused primarily on the heat equation (see, e.g., \cite{AEWZ2014heat-measurable,BurqMoyano2023heat}).
A recent and independent work \cite{NWX2025kdv-control} addresses a related problem for dispersive equations on the one-dimensional torus.

\subsubsection{Eigenfunctions and Quantum Limits}

A closely related topic is the observability for eigenfunctions.
We say that the Schrödinger operator $\hamiltonian_V$ is \emph{observable} from a domain $\omega \subset \torus^d$ if for some $C_{d,V,\omega} > 0$ and for every its eigenfunction $f \in L^2(\torus^d)$, there holds
\begin{equation}
\label{eq::observability-eigen}
    \|f\|_{L^2(\torus^d)} \le C_{d,V,\omega} \|f\|_{L^2(\omega)}.
\end{equation}

Notice that, if the propagator $\schrodinger_V$ is observable from $J \times \omega$ for some $J \subset \mathbb{R}$ of positive measure, then the operator $\hamiltonian_V$ is observable from $\omega$.
The following conjecture is thus weaker than Conjecture~\ref{conj::L-infty-observability}, yet remains open.

\begin{conjecture}
\label{conj::L-infty-observability-eigen}
    On tori of arbitrary dimensions, Schrödinger operators with bounded potentials are observable from domains of positive measure.
\end{conjecture}

Also notice that, if $V=0$ and $J$ is an interval of length $2\pi$, then an orthogonality argument establishes the equivalence between these two types of observability.
Indeed, letting $\Pi_n$ denote the projection on the eigenspace of the free Schrödinger operator $\mathcal{H}_0 = -\Delta$ corresponding to eigenvalue $n \in \mathbb{N}$, Parseval's identity yields
\begin{equation*}
    \|u\|_{L^2(J \times \omega)}^2
    = \int_\omega \int_J \biggl| \sum_{n \ge 0} e^{int} \Pi_n u_0 \biggr|^2 \diff t\diff x
    = 2\pi \sum_{n \ge 0} \|\Pi_n u_0\|_{L^2(\omega)}^2.
\end{equation*}
Surprisingly, this fact seems to have been overlooked in the literature.
    
The following corollary of Theorem~\ref{thm::main-cartesian} extends the results of Zygmund \cite{Zygmund1972cantor-lebesgue} (for $d\le 2$ and $V=0$) and of Connes \cite{Connes1976} (for $V=0$ and open $\omega$).

\begin{corollary}
\label{cor::main-cartesian-eigen}
    The Schrödinger operator $\hamiltonian_V$ is observable from $\omega$ if
    \begin{equation}
    \label{eq::cdt-V-omega-cartesian}
        V \in A, \quad \bm{1}_\omega \in A \setminus \{0\}.
    \end{equation}
    In particular, Conjecture~\ref{conj::L-infty-observability-eigen} holds on tori of dimension at most two.
\end{corollary}

A byproduct of our proof is an observability result for quantum limits.
By definition, a quantum limit associated with $\hamiltonian_V$ is a Radon measure $\mu$ on $\torus^d$ obtained as a weak-$*$ limit of measures $|f_j|^2 \diff x$, where $(f_j)_{j \ge 0}$ is an $L^2$-normalized sequence of eigenfunctions to $\hamiltonian_V$ with increasing eigenvalues.
This means, for all $\phi \in C^0(\torus^d)$, there holds
\begin{equation*}
    \lim_{j \to \infty} \int_{\torus^d} \phi |f_j|^2 \diff x = \int_{\torus^d} \phi \diff \mu.
\end{equation*}
Recall that, by Bourgain \cite{Bourgain1997lattice} (see also \cite{Jakobson1997quantum,Aissiou2013QL}), any quantum limit associated with the free Schrödinger operator is absolutely continuous with respect to the Lebesgue measure on the corresponding torus.
An abstract argument by Burq \cite{Burq2013inhomogeneous} extends this result to Schrödinger operators with bounded potentials.
Therefore, if $V \in L^\infty(\torus^d)$ and if $\mu$ is a quantum measure associated with $\hamiltonian_V$, then $\diff \mu = f \diff x$ for some density $f \in L^1(\torus^d)$.

We show that, in the above limit, one may replace $\phi$ with any time-independent element in the algebra $A$.
Consequently, Corollary~\ref{cor::main-cartesian-eigen} implies the following observability estimate for quantum limits.

\begin{theorem}
    If $V$ and $\omega$ satisfy \eqref{eq::cdt-V-omega-cartesian}, then there exists $C_{d,V,\omega} > 0$ such that, for any quantum limit $\mu$ with density $f$ associated with $\hamiltonian_V$, there holds
    \begin{equation*}
        \mu(\omega) 
        = \int_{\torus^d} f \bm{1}_\Omega \diff x
        \ge C_{d,V,\omega}.
    \end{equation*}
    In particular, when $d \le 2$, this estimate holds for all bounded potentials and for all domains with positive measure.
\end{theorem}

One naturally asks if the above estimate remains valid when $\bm{1}_\Omega \diff x$ is replaced with a Borel measure $\nu$ on $\torus^d$.
In our recent collaboration with Germain and Sorella \cite{BGSZ-trace-observability}, this question is given an affirmative answer when $\nu$ admits certain Fourier decay estimates.

\subsubsection{Cantor--Lebesgue Theorems}

Another motivation of studying observability estimates lies in their connection with \emph{Cantor--Lebesgue} type theorems.

One general formulation of such theorems states that, for a suitable sequence $(f^n)_{n \ge 0}$ of square-integrable functions on a finite measure space (the simplest example being a sequence of Laplacian eigenfunctions on tori), its local pointwise convergence to zero implies its strong $L^2$-convergence to zero.
It turns out that, simple application of Egorov's theorem, this statement is equivalent to the observability of the sequence from subsets of positive measure (Cooke \cite{Cooke1979CL-review}).
Specifically, the following statements are equivalent :
\begin{enumerate}[label=(\roman*)]
    \item For every subset $\omega$ of positive measure, the sequence is observable from $\omega$.
    This means, for some $C_\omega > 0$ and for all $n \ge 0$, there holds
    \begin{equation*}
        \|f^n\|_{L^2} \le C_\omega \|f^n\|_{L^2(\omega)}.
    \end{equation*}
    \item For every subset $\omega$ of positive measure, if the sequence converges pointwisely to zero on $\omega$, then it converges strongly to zero in $L^2$.
\end{enumerate}
The following corollary follows from this equivalence and Theorem~\ref{thm::main-cartesian}.

\begin{corollary}
    The following Cantor--Lebesgue type statements hold:
    \begin{enumerate}
        \item Suppose that $d\le 2$ and $V \in L^\infty(\torus^d)$.
        If $(f^n)_{n \ge 0}$ is a sequence of eigenfunctions to $\hamiltonian_V$ which pointwisely converges to zero on a domain of positive measure, then it converges to zero strongly in $L^2(\torus^d)$.
        \item Suppose that $d=1$ and $V \in L^\infty(\torus)$.
        If $(f^n)_{n \ge 0}$ is a sequence in $L^2(\torus)$ such that $(\schrodinger_V f^n)_{n \ge 0}$ converges to zero on a space-time domain of positive measure, then $(f^n)_{n \ge 0}$ converges to zero strongly in $L^2(\torus)$.
    \end{enumerate}
\end{corollary}

In this corollary, the first statement extends Zygmund's on the Cantor--Lebesgue theorem for Laplacian eigenfunctions on the two-dimensional torus \cite{Zygmund1972cantor-lebesgue} to Schrödinger operators with bounded potentials.
In fact, our main criterion shows that it suffices for the potentials to be slightly better than square integrable.

The second statement, when applied to the free Schrödinger propagator, can also be deduced from Zygmund's proof, though this observation appears to be absent from the literature.
It also yields a uniqueness result, which states that a complex-valued sequence $(c_k)_{k \in \mathbb{Z}}$ vanishes identically if the set of all $(t,x) \in \mathbb{R} \times \mathbb{T}$ where the following limit of partial sums holds, has positive measure:
\begin{equation*}
    \lim_{n \to \infty} \sum_{|k| \le n} c_k e^{i (kx + t k^2)} = 0.
\end{equation*}

\subsection{Criterion for Observability}
\label{sec::main-result-intro-criteria}

We now present stronger criteria than Theorem~\ref{thm::main-cartesian} and Corollary~\ref{cor::main-cartesian-eigen}.
The most general theorem is deferred to \S\ref{sec::main-results} since it involves further definitions and notations.

\subsubsection{Criterion for Schrödinger Waves}

Replacing $\bm{1}_\Omega \diff x \diff t$ with $|\chi|^2 \diff x \diff t$ where $\chi \in L^2(\mathbb{R} \times \torus^d) \setminus \{0\}$, we extend the observability \eqref{eq::observability} to a more general estimate:
\begin{equation}
    \label{eq::observability-generalized}
    \|u_0\|_{L^2(\torus^d)} \le C_{d,V,\chi} \bigl\|\chi \schrodinger_V u_0\bigr\|_{L^2(\mathbb{R} \times \torus^d)}.
\end{equation}
We say that the Schrödinger propagator $\schrodinger_V$ is \emph{$\chi$-observable} if this estimate holds.

Since we aim to establish the observability from all domains of positive measure, we may assume, without losing generality, that $\chi$ satisfies the support condition
\begin{equation}
\label{eq::cdt-chi-supp}
    \supp \chi \subset (-\pi,+\pi) \times \torus^d.
\end{equation}
This assumption allows us to extend $\chi$ periodically to a function in $L^2(\torus_t \times \torus^d_x)$, and we will identify $\chi$ with this periodic extension when it does not create ambiguity.
Here and in the remaining of the paper, for the sake of clarity, we use the notations $\torus_t$ for the temporal torus, and $\torus^d_x$ for the spatial torus $\torus^d$ used previously.

From existing results (which will be reviewed in \S\ref{sec::history}), the observability estimate \eqref{eq::observability-generalized} holds when $V$ and $\chi$ are both continuous.
We will extend this statement to cases where $V$ and $\chi$ are well-approximated by continuous functions.

To quantify this approximation, let
\begin{equation*}
    \mathscr{S}
    = \schrodinger_0 (L^2(\torus^d_x))
    = \{ u : t \mapsto e^{it\Delta} u_0 : u_0 \in L^2(\torus^d_x) \}
\end{equation*}
be the space of square-integrable free Schrödinger waves.
Since free Schrödinger waves are time-periodic, we may further regard $\mathscr{S}$ as a subspace of $L^2(\torus_t \times \torus^d_x)$ from which it inherits the complex Hilbert space structure.
Regarding elements of $L^\infty(\torus_t \times \torus^d_x)$ as bounded multiplication operators on $L^2(\torus_t \times \torus^d_x)$, one has continuous embeddings
\begin{equation*}
    C^0(\torus_t \times \torus^d_x) 
    \hookrightarrow L^\infty(\torus_t \times \torus^d_x) 
    \hookrightarrow Y \coloneqq \mathfrak{L}(\mathscr{S},L^2(\torus_t \times\torus^d_x)),
\end{equation*}
where $\mathfrak{L}(\cdot,\cdot)$ denotes spaces of bounded linear operators between complex Banach spaces.
This allows to define the closure of $C^0(\torus_t \times \torus^d_x)$ in $Y$, which is denoted by $Y_0$, i.e.,
\begin{equation*}
    Y_0 = \overline{C^0(\torus_t \times \torus^d_x)}^Y.
\end{equation*}
The space $Y_0$ will replace the role of $A$ in our criterion.

Notice that, if $u_0 \equiv 1$, then $\schrodinger_0 u_0 \equiv 1$.
Therefore $C_d \|\phi \|_Y \ge \|\phi\|_{L^2(\torus_t \times \torus^d_x)}$ for some $C_d > 0$ and for all $\phi \in C^0(\torus_t \times \torus^d_x)$.
This yields the embedding
\begin{equation*}
    Y_0 \hookrightarrow L^2(\torus_t \times \torus^d_x).
\end{equation*}
Also notice that, the spaces of smooth functions $C^\infty(\torus_t \times \torus^d_x)$ and trigonometric polynomials $\mathscr{T}(\torus_t \times \torus^d_x)$ share the same closure as $C^0(\torus_t \times \torus^d_x)$ in $Y$, i.e., $Y_0$.
Indeed, by the Stone--Weierstrass theorem, they are both dense in $C^0(\torus_t \times \torus^d_x)$ under the $L^\infty$ topology.
Due to the embedding $L^\infty(\torus_t \times \torus^d_x) \hookrightarrow Y$, the same holds true under the $Y$-topology.

We will also relax the integrability condition for the potential $V$.
For this purpose, we define the intervals of Lebesgue exponents:
\begin{equation*}
    I_1 = [1,\infty];
    \quad
    I_2 = (1,\infty];
    \quad
    I_d = [d/2,\infty],\ \forall d \ge 3.
\end{equation*}
Recall from classical elliptic theory that taking $V \in L^q(\torus^d_x)$ with $q \in I_d$ ensures the Schrödinger operator $\hamiltonian_V$ to be essentially self-adjoint.
This condition also guarantees that $\schrodinger_V$ satisfies a unique continuation property from subsets of positive measure (Theorem~\ref{thm::unique-continuation}), a key ingredient in our proof.

Our criterion for the observability of Schrödinger operators states as follows.

\begin{theorem}
    \label{thm::main-general-bis}
    The Schrödinger propagator $\schrodinger_V$ is $\chi$-observable if, for some $q \in I_d$,
    \begin{equation}
    \label{eq::V-chi-condition-almost}
        V \in L^q(\torus^d_x) \cap Y_0,
        \quad
        \chi \in Y_0 \setminus \{0\}.
    \end{equation}
\end{theorem}

By this theorem, all the aforementioned results, including the applications to control theory and quantum limits, remain valid when the assumptions \eqref{eq::cdt-V-Omega-A} are relaxed to \eqref{eq::V-chi-condition-almost} (i.e., when $A$ is replaced with $Y_0$).
In view of the assumption \eqref{eq::V-chi-condition-almost}, we propose the following conjecture, which is stronger than Conjecture~\ref{conj::L-infty-observability} and hence also Conjecture~\ref{conj::L-infty-observability-eigen}.

\begin{conjecture}
    \label{conj::Y-sup-L-infty}
    For all $d \ge 1$, there holds
    \begin{equation}
    \label{eq::Y_0-contains-L-infty}
        Y_0 \supset L^\infty(\torus_t \times \torus^d_x).
    \end{equation}
\end{conjecture}

It follows from Zygmund \cite{Zygmund1974Fourier} that if $d=1$ then $\mathscr{S} \hookrightarrow L^4(\torus_t \times \torus_x)$, which then implies $L^4(\torus_t \times \torus_x) \hookrightarrow Y_0$ by Hölder's inequality.
Consequently, this conjecture holds on the one-dimensional torus.
Even though it remains open in higher dimensions, we have the following theorem, which immediately implies Theorem~\ref{thm::main-cartesian}.

\begin{theorem}
\label{thm::Y_0-algebra-subset}
    The space $Y_0 \cap L^\infty(\torus_t \times \torus^d_x)$ is a closed subalgebra of $L^\infty(\torus_t \times \torus^d_x)$ and contains all bounded function on $\torus_t \times \torus^d_x$ that are either Riemann integrable or dependent on at  most two of the space-time variables $t, x_1,\ldots,x_d$.
\end{theorem}

The fact that bounded functions depending on at most two variables belong to $Y_0$ is a natural extension of Zygmund's estimate \cite{Zygmund1974Fourier} to higher dimensions (Lemma~\ref{lem::Xb-embedding-one-two}), while functions depending on three or more variables remain unaccounted for, due to the absence of a similar estimate when the number of variables is at least three.

To show that Riemann integrable functions are also included (see \S\ref{sec::riemann}), we utilize the fact that any semiclassical limit of free Schrödinger waves is absolutely continuous with respect to the Lebesgue measure on the corresponding torus \cite{Bourgain1997lattice,Anantharaman-Macia2014measure,Burq2013inhomogeneous}.

\subsubsection{Criterion for Eigenfunctions}

In parallel to the Schrödinger propagator $\schrodinger_V$, we also define the $\chi$-observability for the Schrödinger operator $\hamiltonian_V$ when $\chi \in L^2(\torus^d_x)$.
This is simply the following estimate which generalizes \eqref{eq::observability-eigen}:
\begin{equation}
    \|f\|_{L^2(\torus^d_x)} \le C_{d,V,\chi} \|\chi f\|_{L^2(\torus^d_x)}.
\end{equation}
By Theorem~\ref{thm::main-general-bis}, if $V$ and $\chi$ meet the assumption \eqref{eq::V-chi-condition-almost}, then $\hamiltonian_V$ is $\chi$-observable.
This leads to the following weaker conjecture than Conjecture~\ref{conj::L-infty-observability-eigen}:
\begin{equation}
\label{eq::Y_0-contains-L-infty-no-time}
    Y_0 \supset L^\infty(\torus^d_x).
\end{equation}

\subsection{Equivalent Formulations}

There exist several equivalent formulations of the conjectured relations \eqref{eq::Y_0-contains-L-infty} and \eqref{eq::Y_0-contains-L-infty-no-time}.
For this purpose, 
\begin{itemize}
    \item let $\mathscr{S}_*$ be the set of all $L^2$-normalized free Schrödinger waves; and
    \item let $\mathscr{E}_*$ be the set of all $L^2$-normalized eigenfunctions of $\Delta$.
\end{itemize}
For any set of functions $S$, we denote $|S|^2 = \{|f|^2 : f \in S\}$.
Also recall that a convex, even, and lower semicontinuous function $\Phi : \mathbb{R} \to [0,\infty)$ is a \emph{strict Young function} if its Legendre transform admits finite values everywhere or, equivalently, if there holds
\begin{equation*}
    \lim_{s\to \infty} |s|^{-1} \Phi(s) = \infty.
\end{equation*}

\subsubsection{Relations with Uniform Integrability}
\label{sec::intro-relation-UI}

We now state these equivalent formulations.

\begin{theorem}
    \label{thm::equivalence-UI-schrodinger}
    The following statements are equivalent:
    \begin{enumerate}
        \item $Y_0 \supset L^\infty(\torus_t \times \torus^d_x)$;
        \item $\mathscr{S}_*$ is uniformly square integrable;
        \item $|\mathscr{S}_*|^2$ is relatively compact with respect to the weak topology $\sigma(L^1,L^\infty)$;
        \item There exists a strict Young function $\Phi$ such that $\Phi(|\mathscr{S}_*|^2) = \{\Phi(|u|^2) : u \in \mathscr{S}_*\}$ is a bounded family in $L^1(\torus_t \times \torus^d_x)$.
    \end{enumerate}
In parallel, the following statements are also equivalent:
    \begin{enumerate}
        \item $Y_0 \supset L^\infty(\torus^d_x)$;
        \item $\mathscr{E}_*$ is uniformly square integrable;
        \item $|\mathscr{E}_*|^2$ is relatively compact with respect to the weak topology $\sigma(L^1,L^\infty)$.
        \item There exists a strict Young function $\Phi$ such that $\Phi(|\mathscr{E}_*|^2) = \{\Phi(|f|^2) : f \in \mathscr{E}_*\}$ is a bounded family in $L^1(\torus^d_x)$.
    \end{enumerate}
\end{theorem}

In each set of formulations, the equivalence between the last three statements follows from the Dunford--Pettis theorem (see e.g., \cite{Bogachev:Book}) and the De La Vallée--Poussin theorem \cite{DeLaValleePoussin1915UI}.
Recall that, by Bourgain \cite{Bourgain1997lattice} and Jakobson \cite{Jakobson1997quantum} (which is essentially also due to Bourgain) (see also \cite{Anantharaman-Macia2014measure,Burq2013inhomogeneous,AissiouJakobsonMacia2012} for cases with potentials), the families $|\mathscr{S}_*|^2$ and $|\mathscr{E}_*|^2$ are relatively compact in with respect to the weak-$*$ topologies (on the spaces of Radon measures), which are weaker than the weak topologies.

\subsubsection{Relations with Classical Integrability Bounds}

The boundedness of $\Phi(|\mathscr{S}_*|^2)$ in $L^1(\torus_t \times \torus^d_x)$ can be regarded as a weaker version of Bourgain's conjectured periodic Strichartz estimate \cite{Bourgain1997lattice}, which claim that, if $0 < p < \frac{2(d+2)}{d}$, then for some $C_{d,p} > 0$ and for all $u_0 \in L^2(\torus^d_x)$, there holds
\begin{equation*}
    \|e^{it\Delta} u_0\|_{L^p(\torus_t \times \torus^d_x)} \le C_{d,p} \|u_0\|_{L^2(\torus^d_x)}.
\end{equation*}
On the other hand, the boundedness of $\Phi(|\mathscr{E}_*|^2)$ in $L^1( \torus^d_x)$ can be regarded as a weaker version of Bourgain's conjectured eigenfunction bounds \cite{Bourgain1993eigenfunction,Bourgain1997lattice} (see also \cite{BourgainDemeter2015decoupling} for later development), which claim that, if $d\ge 2$ and $0 < p < \frac{2d}{d-2}$, then for all every eigenfunction $f$ to the toral Laplacian, there holds
\begin{equation*}
    \|f\|_{L^p(\torus^d_x)} \lesssim_{d,p} \|f\|_{L^2(\torus^d_x)}.
\end{equation*}

Indeed, if any of these estimates hold for some $p > 2$, then the corresponding uniform square integrability follows by choosing $\Phi(s)=|s|^{p/2}$.
Recall that, by Zygmund \cite{Zygmund1974Fourier}, the Strichartz estimate holds holds for $d=1$ and $p=4$ whereas the eigenfunction bound holds for $d=2$ and $p=4$.

\subsection{Strategy of Proof}
\label{sec::strategy-of-proof}

In \S\ref{sec::main-results}, we state our main result (Theorem~\ref{thm::main-general}), which further generalizes Theorem~\ref{thm::main-general-bis}.
Its proof is based on \emph{cluster decomposition} and \emph{dimensional reduction}, inspired by the works of Connes \cite{Connes1976} and Bourgain \cite{Bourgain2014control-3D}.

\subsubsection{Cluster Decomposition}

For eigenfunctions, one uses the cluster structure for integer points on spheres.
This structure dates back to Jarník \cite{Jarnik1925konvexen} and Connes \cite{Connes1976} with later development such as \cite{CillerueloCordoba1992trigo,Bourgain1997lattice,BourgainRudnick2011nodal,BGSZ-trace-observability}.
For Schrödinger waves, one uses the cluster structure for integer points on the parabolic characteristic set, i.e., one considers
\begin{equation}
    \label{eq::def-characteristic-set}
    \Sigma = \{(-|k|^2, k) : k \in \mathbb{Z}^d\} \subset \mathbb{Z} \times \mathbb{Z}^d.
\end{equation}
When $V\ne 0$, one consider a neighborhood of this set.

This structure, due to Granville and Spencer (see e.g., \cite{Bourgain1997lattice,Bourgain1998quasi-periodic,Bourgain2005lattice}), allows us to use the Fourier method.
Nevertheless, our proof is deeply connected to the propagation approach in \cite{Macia2010propagation,Macia2011observability,BurqZworski2012-2d-smooth,BourgainBurqZworski2013rough-potential,Anantharaman-Macia2014measure,AnantharamanMacia2012dynamics}.
We believe that the cluster structure is closely linked to the second microlocal characterization of wave propagation used by Anantharaman and Macià \cite{Anantharaman-Macia2014measure}.
This connection will be further revealed in our forthcoming work in collaboration with Camps and Sun \cite{BCSZ-cluster-microlocalization}.

\subsubsection{Mathematical Induction on Dimensions}

Although inspired by Bourgain's clustering approach \cite{Bourgain2014control-3D} on $\mathbb{T}^3$, we pursue a \emph{different strategy}: we exploit the fact that every cluster, except for a finite number of them, projects onto a lower-dimensional affine sublattice of $\mathbb{Z}^d$.  
This geometric insight enables a dimensional reduction argument and an induction-based proof.

Due to the large separation between clusters, the observability estimate \eqref{eq::observability-generalized} effectively \emph{decouples} when $V$ and $\chi$ are smooth or well-approximated by smooth functions, reducing to estimates on cluster-localized Fourier modes.  

\subsubsection{Modified Uniqueness-Compactness Argument}

The above scheme only yields a \emph{weak observability} since, as previously mentioned, the projections of some clusters may have full space dimension.
To recover the strong observability, we apply the \emph{uniqueness-compactness argument} of Bardos, Lebeau, and Rauch \cite{BardosLebeauRauch1992}.  

Precisely, when $\chi$ is a product of a continuous temporal function and a measurable spatial function, this classical argument applies directly.
Otherwise, \emph{two additional ingredients} are required:
\begin{enumerate}[label=(\roman*)]
    \item To ensure uniformity across affine sublattices, we use \emph{Galilean transformations} to shift high-frequency clusters toward the low-frequency region.
    \item When $\Omega$ is not a product of a time interval and a spatial domain, we appeal to a \emph{temporal holomorphic extension} and the Lusin--Privalov uniqueness theorem which reduce the problem to a classical setting.
\end{enumerate}

\subsection{Historical Review and Comparison}
\label{sec::history}

Let us review previous and recent results on observability and related problems.
We also compare these developments with some results obtained from our criterion (which will be given in \S\ref{sec::main-results}).
See Tables~\ref{table::past-results-eigenfunction} and~\ref{table::past-results-schrodinger}.

\subsubsection{Toral Eigenfunctions}

For observability and Cantor--Lebesgue theorems concerning eigenfunctions of the Laplacian on tori, early results trace back to Cantor \cite{Cantor1870trigono}, Cooke \cite{Cooke1971}, Zygmund \cite{Zygmund1972cantor-lebesgue,Zygmund1974Fourier}, and Connes \cite{Connes1976}.

Bourgain and Rudnick \cite{BourgainRudnick2012restriction} investigated observability from curved hypersurfaces.
With Germain and Sorella \cite{BGSZ-trace-observability}, we extend their results to Borel measures with Fourier decay.

See also \cite{Nazarov1993,FontesMerz2006,Egidi-Veselic2020,Tao2021small,GermainMoyanoZhu2024vanishing} for quantitative results.

\begin{table}[htb!]\small
    \begin{tabular}{c|c|c|c}
        Authors \& References & Dimension $d$  & Potential $V$ & Domain $\omega$ \\
        \hline
        Cantor \cite{Cantor1870trigono} & $=1$ & $= 0$ & $|\omega| > 0$ \\
        Zygmund \cite{Zygmund1972cantor-lebesgue} & $=2$ & $= 0$ & $|\omega| > 0$ \\
        Connes \cite{Connes1976} & $\ge 1$ & $= 0$ & open\\
        Bourgain--Rudnick \cite{BourgainRudnick2012restriction} & $=2,3$ & $= 0$ &  curved hypersurface \\
        Burq--Germain--Sorella--Zhu \cite{BGSZ-trace-observability} & $\ge 1$ & $=0$ &  Fourier decay \\
        \hline
         & $=1$ & $\in L^2$ &  $|\omega| > 0$ \\
        Burq--Zhu & $=2$ & $\in L^{2+}$ &  $|\omega| > 0$ \\
         & $\ge 1$ & $\in L^q \cap Y_0$ &  $\bm{1}_\omega \in Y_0 \setminus \{ 0 \}$ \\
        \hline
    \end{tabular}
    \caption{Observability for toral eigenfunctions.}
    \label{table::past-results-eigenfunction}
\end{table}

\subsubsection{Toral Schrödinger Waves}

For toral Schrödinger propagators, known results are mostly within the context of control theory.

Briefly, Haraux \cite{Haraux1989lacunaires}, Jaffard \cite{Jaffard1990plaque} and Komornik \cite{Komornik1992Petrowsky} used Kahane's theory of lacunary series \cite{Kahane1962lacunary}, while Macià and Anantharaman \cite{Macia2010propagation,Macia2011observability,Anantharaman-Macia2014measure} used microlocal methods.

For bounded $V$ and measurable $\Omega$, known results are restricted to $d \le 3$ and were obtained in collaborated and/or individual works by Bourgain, Burq, and Zworski \cite{BurqZworski2012-2d-smooth,BourgainBurqZworski2013rough-potential,Bourgain2014control-3D,BurqZworski2019rough}, via a combination of the Fourier method and microlocal analysis.

See also \cite{Tao2021small,SuSunYuan2025-1d} for quantitative results.

From Table~\ref{table::past-results-schrodinger}, one sees that our results have covered and further extended all existing results, with two exceptions:
\begin{enumerate}[label=(\roman*)]
    \item In \cite{BourgainBurqZworski2013rough-potential} where $d=2$ and $\Omega$ is open, the observability was established for $V \in L^2_x$, which is weaker than our assumption $V \in L^{2+}_x$.
    While we are considering more general domains that are not necessarily open, we believe it should be possible to relax the integrability condition on the potential to require only square integrability.
    
    One potential approach to addressing this problem is to use the $U^p$ and $V^p$ spaces introduced by Koch and Tataru \cite{KochTataru2007spaces}, rather than Bourgain's Fourier restriction spaces, which is the approach we adopt in this paper.
    However, for the sake of simplicity, we refrain from pursuing this alternative here.
    
    \item In \cite{Bourgain2014control-3D} where $d=3$, the observability was established when $V$ is bounded and $\Omega$ is open.
    Although this case falls within Conjecture~\ref{conj::Y-sup-L-infty}, we do not yet know how to derive it from our criterion.
\end{enumerate}

\begin{table}[htb!]\small
    \begin{tabular}{c|c|c|c}
        Authors \& References & Dimension $d$ & Potential $V$ & Domain $\Omega$ \\
        \hline
        Haraux \cite{Haraux1989lacunaires}, Jaffard \cite{Jaffard1990plaque} & $=2$ & $=0$ & open\\
        Komornik \cite{Komornik1992Petrowsky}, Macià \cite{Macia2010propagation} & $\ge 1$ & $=0$ & open  \\
        Burq--Zworski \cite{BurqZworski2012-2d-smooth} & $=1,2$ & $\in C^\infty$ & open\\
        Anantharaman--Macià \cite{Anantharaman-Macia2014measure} & $\ge 1$ & Riemann integrable & open\\
        Bourgain--Burq--Zworski \cite{BourgainBurqZworski2013rough-potential} & $=1,2$ & $\in L^2$ & open\\
        Bourgain \cite{Bourgain2014control-3D} & $=3$ & $\in L^\infty$ & open \\
        Burq--Zworski \cite{BurqZworski2019rough} & $=1,2$ & $=0$ & $[0,T] \times \omega$, $|\omega| > 0$\\
        \hline
         & $=1$ & $\in L^2$ & $|\Omega| > 0$ \\
        Burq--Zhu & $= 2$ & $\in L^{2+}$ &  $J \times \omega$, $|J| > 0$, $|\omega| > 0$ \\
         & $\ge 1$ & $\in L^q \cap Y_0$ &  $\bm{1}_\Omega \in Y_0 \setminus \{ 0 \}$ \\
        \hline
    \end{tabular}
    \caption{Observability for toral Schrödinger waves.}
    \label{table::past-results-schrodinger}
\end{table}

All the above-mentioned works rely on the following deeply interconnected ingredients which fundamentally stems from the topology of the torus:
\begin{enumerate}
    \item \emph{Lacunarity.} 
    Fourier modes of Schrödinger waves exhibit lacunarity. 
    Estimates akin to Ingham's one-dimensional inequality \cite{Ingham1936inequality} are expected to hold in higher dimensions.
    One of such theories was developed by Kahane \cite{Kahane1962lacunary} and employed in early studies of the free propagator \cite{Haraux1989lacunaires,Jaffard1990plaque}.

    \item \emph{Wave propagation.}
    The use of wave propagation in observability dates back to Rauch and Taylor \cite{RauchTaylor1975decay} and Bardos, Lebeau and Rauch \cite{BardosLebeauRauch1992}.
    Second microlocal versions of this idea have successful applications \cite{Macia2011observability,Anantharaman-Macia2014measure,AFKM2015integrable,BurqGerard2020damping}.

    \item \emph{Dimensional reduction}.
    The reduction of estimates in arbitrary dimensions to the ones in lower dimensions is a particularly useful technique on tori and has been used in (almost) all previously mentioned works.
\end{enumerate}

\subsubsection{General Settings}

The observability was broadly investigated in various dispersive and geometric settings where $V=0$ and $\Omega = [0,T] \times \omega$ with $T > 0$ and $\omega$ open.

On general manifolds, a \emph{geometric control condition} (GCC), which requires that all geodesics intersect $\omega$, is a sufficient and nearly necessary condition for observability of the wave equation; see e.g., \cite{RauchTaylor1975decay,BardosLebeauRauch1992,BurqGerard1997wave,BurqGerard2020damping}.

For the Schrödinger equation, it is sufficient but unnecessary on tori (Table~\ref{table::past-results-schrodinger}) or hyperbolic surfaces (Jin \cite{Jin2018control}); it is almost necessary in cases such as spheres (Lebeau \cite{Lebeau1992schrodinger}).
For fractional Schrödinger equations, GCC is sufficient (Macià \cite{Macia2021fractional}) and seems necessary on tori (see Zhu \cite{Zhu2020controlWW} for $d=2$).

In Euclidean spaces, observability estimates of the Schrödinger propagator hold from certain domains violating GCC (Wunsch~\cite{Wunsch2017periodic} and Täufer \cite{Taeufer2023control-nonGCC}).
When $d=2$, these results were extended to a rough setting (Le Balch'H and Martin \cite{BalchMartin2023obs}).
For related works on nonlinear Schrödinger equations, see e.g., \cite{Laurent2010control,Laurent2010manifold,BeauchardLaurent2010bilinear,IandoliNiu2023control,SilvaEtAl2024control}.

\subsection{Paper Organization and Notations}

In \S\ref{sec::main-results}, we state our criterion for observability in full generality and present several examples derived from it.
In \S\ref{sec::proof-free-propagator} and \S\ref{sec::proof-general-propagator}, we establish the criterion for the free propagator and for the general propagator, respectively.
Finally, in \S\ref{sec::unique-continuation}, we prove a unique continuation result from measurable subsets.

The following notations are used throughout the paper for the simplicity of presentation.
For function spaces on tori, we write
\begin{equation*}
    Z_t = Z(\torus_t),
    \quad
    Z_x = Z(\torus^d_x),
    \quad
    Z_{t,x} = Z(\torus_t \times \torus^d_x),
\end{equation*}
where $Z$ varies from
\begin{enumerate*}[label=(\roman*)]
    \item $C^0$, the space of continuous functions,
    \item $C^\infty$, the space of smooth functions,
    \item $\mathscr{T}$, the space of trigonometric polynomials,
    \item $\mathscr{D}'$, the space of distributions,
    \item $L^p$, the Lebesgue spaces,
    \item $H^b$, the Sobolev spaces, etc.
\end{enumerate*}

We also denote by $Z_{t,+}$, $Z_{x,+}$, and $Z_{t,x,+}$ the corresponding subsets that consist of all nonnegative functions, whenever the notion of nonnegativity is well-defined. 

For any set of parameters $\mathscr{A}$, we write $f \lesssim_{\mathscr{A}} g$ resp.\ $f \propto_{\mathscr{A}} g$ if $f \le C_{\mathscr{A}} g$ resp.\ $f = C_{\mathscr{A}} g$ for some positive and finite constant $C_{\mathscr{A}}$ depending solely on $\mathscr{A}$.
We also write $f \asymp_{\mathscr{A}} g$ if $f \lesssim_{\mathscr{A}} g$ and $g \lesssim_{\mathscr{A}} f$ holds simultaneously.

For any set $S$, we denote by $\sharp S$ its cardinality.

For any Banach spaces $E$ and $F$, we write $E \hookrightarrow F$ if $E$ embeds continuously into $F$.

\subsection{Acknowledgement}

The research of the authors has received funding from the European Research Council (ERC) under the
European Union's Horizon 2020
research and innovation programme (Grant agreement 101097172 -- GEOEDP). 

The authors would like to thank Pierre Germain, Massimo Sorella, Chenmin Sun, Nikolay Tzvetkov, and Claude Zuily for helpful discussions.

\section{Main Results and Examples}
\label{sec::main-results}

In this section, we state our main theorems for observability and their applications.

First, we fix the definition for the space-time Fourier transform:
\begin{equation*}
    \widehat{f}(n,k) = (2\pi)^{-(1+d)/2} \iint e^{-i(nt+k\cdot x)} f(t,x) \diff t \diff x.
\end{equation*}
If $f$ only depends on $t$ resp.\ $x$, then $\widehat{f}$ is identified with its temporal resp.\ spatial Fourier transform.
For $b \in \mathbb{R}$, the (periodic) Bourgain space  $X^b \coloneqq X^{0,b}$ (see e.g., \cite{Bourgain1993fourier-schrodinger,Ginibre1996Bourbaki,Tao2006dispersive}) is the space of all distributions $f \in \distribution_{t,x}'$ such that
\begin{equation*}
    \|f\|_{X^{b}} 
    \coloneqq \|e^{-it \Delta} f(t)\|_{H^b_t L^2_x}
    = \|\langle n + |k|^2 \rangle^b  \widehat{f}(n,k)\|_{\ell^2_n \ell^2_k} < \infty.
\end{equation*}
One clearly has the duality $(X^b)^* \simeq X^{-b}$.
To simplify the presentation, we will also denote $X^\infty = \mathscr{S}$ and denote by $X^{-\infty}$ its dual space.

We also consider bounded linear operators between Bourgain spaces, and denote
\begin{equation*}
    Y^{b,b'} = \mathfrak{L}(X^b,X^{b'}).
\end{equation*}
When $\infty \ge b \ge 0 \ge b' \ge -\infty$, there holds the continuous embeddings
\begin{equation*}
    C^0_{t,x} \hookrightarrow L^\infty_{t,x} \hookrightarrow \mathfrak{L}(L^2_{t,x},L^2_{t,x}) = Y^{0,0} \hookrightarrow Y^{b,b'}.
\end{equation*}
This allows us to define $Y^{b,b'}_0$ and $Y^{b,b'}_{0,+}$ as closures of $C^0_{t,x}$ and $C^0_{t,x,+}$ in $Y^{b,b'}$, that is
\begin{equation*}
    Y^{b,b'}_0 = \overline{C^0_{t,x}}^{Y^{b,b'}},
    \quad
    Y^{b,b'}_{0,+} = \overline{C^0_{t,x,+}}^{Y^{b,b'}}.
\end{equation*}

\subsection{Main Criterion}

We are now ready to state our result in its fullest generality.

\begin{theorem}
    \label{thm::main-general}
    The Schrödinger propagator $\schrodinger_V$ is $\chi$-observable if, for some $q \in I_d$, $b \in (1/2,1)$, and $\epsilon > 0$, there holds
    \begin{equation}
        \label{eq::V-chi-property}
        V \in L^q_x \cap Y^{b,b-1+\epsilon}_0,
        \quad
        |\chi|^2 \in Y^{b,-b}_{0,+} \setminus \{0\}.
    \end{equation}
\end{theorem}

The assumptions \eqref{eq::V-chi-property} are \emph{closed} conditions for they are preserved under limits.
In contrast, the validity of the observability estimate \eqref{eq::observability-generalized} is an \emph{open} condition since it is stable under suitable perturbations of $(V,\chi)$ (particularly in $L^\infty_x \times L^\infty_{t,x}$).
Bridging this gap poses is an important question for future investigation.

In the proof of this theorem, we only require a (slightly) weakened condition for $V$:
\begin{equation}
\label{eq::cdt-V-weaker}
    V \in L^q_x \cap Y^{b,b-1}_0 \cap Y^{b,b-1+\epsilon}.
\end{equation}
As mentioned in \S\ref{sec::main-result-intro-criteria}, one needs $V \in L^q_x$ for a unique continuation result.
The mapping property $V \in Y^{b,b-1+\epsilon}$ is used in a contraction estimate (Lemma~\ref{lem::Duhamel-contraction}).

For the free propagator, we may further relax the condition for $\chi$.

\begin{theorem}
\label{thm::main-free}
    The free Schrödinger propagator $\schrodinger_0$ is $\chi$-observable if
    \begin{equation*}
        |\chi|^2 \in Y^{\infty,-\infty}_{0,+} \setminus \{0\}.
    \end{equation*}
\end{theorem}

When $b \in (1/2,\infty]$, the assumption $|\chi|^2 \in Y^{b,-b}_{0,+}$ (which is used in both of the above theorems) is weaker than the assumption $\chi \in Y_0$ (which is used in Theorem~\ref{thm::main-general-bis}).
More accurately, the former is at most as strong as the latter, with no known examples distinguishing the two conditions.
To see this, we show that, if $b > 1/2$, then
\begin{equation}
\label{eq::Yb0=Y}
    Y^{b,0} = Y,
    \quad
    Y^{b,0}_0 = Y_0.
\end{equation}
Indeed, we use $\| e^{it\Delta} u_0 \|_{X^b} \lesssim_b \| u_0 \|_{L^2_x}$ to derive $Y^{b,0} \subset Y$, and then use a classical transfer argument (detailed in Lemma~\ref{lem::ginibre}) to derive that $Y^{b,0} \supset Y$.
Next, from a $T^*T$ argument, we deduce that, for all $\chi \in L^2_{t,x}$, there holds
\begin{equation*}
    \chi \in Y_0 \iff \chi \in Y^{b,0}_0 \implies |\chi|^2 \in Y^{b,-b}_{0,+}.
\end{equation*}

Next, we show that $Y_0 \cap L^\infty_{t,x}$ is a closed subalgebra of $L^\infty_{t,x}$.
The closeness follows from the embedding $L^\infty_{t,x} \hookrightarrow Y$.
To show the stability under multiplication, let $f,g \in Y_0 \cap L^\infty_{t,x}$, and let $(f_j)_{j \ge 0}, (g_j)_{j \ge 0}$ be sequences in $C^0_{t,x}$ that converge to them with respect to the $Y$-norm.
Then, as $j \to \infty$, there holds
\begin{align*}
    \|f_j g_j - fg\|_Y
    & \le \|(f_j - f) g_j\|_Y + \|f (g_j - g)\|_Y \\
    & \le \|g_j\|_{L^\infty_{t,x}} \|f_j-f\|_Y + \|f\|_{L^\infty_{t,x}} \|g_j - g\|_Y \to 0.
\end{align*}
This shows that that $(f_jg_j)_{j \ge 0}$, as a sequence in $C^0_{t,x}$, converges to $fg$ in $Y$.

In the following sections, we will identify certain subsets of $Y_0$ and prove Theorem~\ref{thm::Y_0-algebra-subset}.

\subsection{Convenability}

By Bourgain \cite{Bourgain1993fourier-schrodinger}, the continuous embedding $\mathscr{S} \hookrightarrow L^\infty_t L^2_x$ extends
to another continuous embedding $X^b \hookrightarrow L^\infty_t L^2_x$ for all $b > 1/2$.
This fact was generalized by Ginibre \cite{Ginibre1996Bourbaki} and will be further generalized below.

\begin{definition}
    A bounded linear mapping $\mathcal{A} : \mathscr{S} \to E \subset\distribution'$ with $E$ being a Banach space is called \emph{convenable} if the following conditions are satisfied:
    \begin{enumerate}
        \item  For all $n \in \mathbb{N}$, there holds $[\mathcal{A},e^{int}] = 0$, with $e^{int}$ regarded as multipliers.
        \item For all $n \in \mathbb{N}$, there holds $\|e^{int}\|_{\mathfrak{L}(E,E)} \lesssim_{d,E} 1$.
    \end{enumerate}
    If $\mathscr{S} \hookrightarrow E$ is convenable, then we also call $E$ convenable.
\end{definition}

In \cite{Ginibre1996Bourbaki}, Ginibre defined the convenability via a stronger condition: $\|\eta\|_{\mathfrak{L}(E,E)} \lesssim_E \|\eta\|_{L^\infty_t}$ for all $\eta \in L^\infty_t$ and proved $X^b \hookrightarrow E$ for $b > 1/2$.
Now we extends this claim.

\begin{lemma}
    \label{lem::ginibre}
    If $\mathcal{A} : \mathscr{S} \to E$ is convenable and $b > 1/2$, then $\mathcal{A}$ extends uniquely to a bounded linear operator $\mathcal{A}_b \in \mathfrak{L}(X^b,E)$ with comparable norm.
    Particularly, if $E$ is a convenable, then $X^b \hookrightarrow E$ for $b > 1/2$.
\end{lemma}

\begin{proof}
    For $u \in X^b$, let $v(t) = e^{-it\Delta} u(t)$ and let $v^n(t) = e^{it\Delta} v^n_0$ where $\sqrt{2\pi} v^n_0(x)$ is the temporal Fourier transform $\mathcal{F}_{t\to n}$ of $v(t,x)$, then
    \begin{equation*}
        u(t,x) = \sum_{n \in \mathbb{Z}} e^{int} v^n(t,x),
        \quad
        v(t,x) = \sum_{n \in \mathbb{Z}} e^{int} v^n_0(x).
    \end{equation*}
    To define $\mathcal{A}_b$, by linearity and the commutator relation $[\mathcal{A},e^{int}] = 0$, the only way of formally defining $\mathcal{A}_b u$ when $u \in X^b$ is by setting
    \begin{equation}
        \label{eq::ginibre-A-extension-def}
        \mathcal{A}_b u
        = \sum_{n \in \mathbb{Z}} e^{int} \mathcal{A} v^n.
    \end{equation}
    To show that this is well-defined, we prove the absolute summability of the series.
    Indeed, using $b > 1/2$ and the definition of Bourgain spaces, we have:
    \begin{align*}
        \sum_{n \in \mathbb{Z}} \|e^{int} \mathcal{A} v^n\|_E
        & \lesssim \|\mathcal{A}\|_{\mathfrak{L}(\mathscr{S}, E)} \sum_{n \in \mathbb{Z}} \|v^n\|_{L^2_{t,x}}
        \asymp \|\mathcal{A}\|_{\mathfrak{L}(\mathscr{S}, E)} \sum_{n \in \mathbb{Z}} \|v^n_0\|_{L^2_x} \\
        & \lesssim_b  \|\mathcal{A}\|_{\mathfrak{L}(\mathscr{S}, E)} \Bigl(\sum_{n \in \mathbb{Z}} \langle n \rangle^{2b} \|v^n_0\|_{L^2_x}^2\Bigr)^{1/2}
        \asymp \|\mathcal{A}\|_{\mathfrak{L}(\mathscr{S}, E)}  \|u\|_{X^b}.
    \end{align*}
    This shows that the extension is well-defined with estimate $\|\mathcal{A}\|_{\mathfrak{L}(X^b,E)} \lesssim_b \|\mathcal{A}\|_{\mathfrak{L}(\mathscr{S}, E)}$.
    Since $\mathscr{S} \hookrightarrow X^b$, the reverse direction of the estimate is trivial.

    Finally, one sees immediately from \eqref{eq::ginibre-A-extension-def} that if $\mathcal{A}: \mathscr{S} \to E$ is an embedding, then the extension $\mathcal{A}_b : X^b \to E$ is also an embedding.
\end{proof}

A motivation of finding convenable spaces is that they yield subspaces of $Y_0$.
In fact, if $E$ is convenable, then $\mathscr{S} \hookrightarrow E$.
Therefore,
\begin{equation*}
    \mathfrak{L}(E,L^2_{t,x})
    \hookrightarrow Y = \mathfrak{L}(\mathscr{S},L^2_{t,x}),
    \quad
    \overline{C^0_{t,x}}^{\mathfrak{L}(E,L^2_{t,x})}
    \hookrightarrow Y_0 = \overline{C^0_{t,x}}^{\mathfrak{L}(\mathscr{S},L^2_{t,x})}.
\end{equation*}
Known examples of convenable spaces include the following:
\begin{enumerate}[label=(\roman*)]
    \item $L^\infty_t L^2_x$ for all $d \ge 1$ (conservation of energy);
    \item $L^4_{t,x}$ for $d=1$ (Zygmund's $L^4$ estimate \cite{Zygmund1974Fourier}), and
    \item $L^4_x L^2_t$ for $d=2$ (Bourgain, Burq, and Zworski \cite{BourgainBurqZworski2013rough-potential}).
\end{enumerate} 
These examples extends to higher dimensions.
To sate these extensions, write
\begin{equation*}
    z = (z_0,z_1,\ldots,z_d) = (t,x_1,\ldots,x_d).
\end{equation*}
For $\alpha \subset \{0,1,\dots,d\}$, let $z_\alpha = (z_j)_{j \in \alpha} \in \torus^\alpha$ and let $\zeta_\alpha = (\zeta_j)_{j \in \alpha} \in \mathbb{Z}^\alpha$ be the corresponding dual (Fourier) variables.

\begin{lemma}
    \label{lem::Xb-embedding-one-two}
    Suppose that $(\sharp \alpha,p,q) = (1,\infty,2)$ or $(\sharp \alpha,p,q) = (2,4,4)$, so that
    $ \frac{1}{p} + \frac{1}{q} = \frac{1}{2}$.
    Next, let $\beta = \{0,1,\ldots,d\} \backslash \alpha$, then $L^p(\torus^\alpha_{z_\alpha};L^2(\torus^{\beta}_{z_\beta}))$ is convenable, and consequently
    \begin{equation}
        \label{eq::Y-embedding-examples}
        L^q(\torus^\alpha_{z_\alpha}; L^\infty(\torus^\beta_{z_\beta}))
        \hookrightarrow Y,
        \quad
        L^q(\torus^\alpha_{z_\alpha}; C^0(\torus^\beta_{z_\beta}))
        \hookrightarrow Y_0.
    \end{equation}
    Therefore, the space $Y_0$ contains all bounded functions that are dependent on at most two of the space-time variables $t,x_1,\ldots,x_d$
\end{lemma}

\begin{proof}
    The first embedding follows from the convenability whereas the second follows from approximation.
    To prove the convenability, we identify
    \begin{equation*}
        \torus^{1+d} = \torus^\alpha \times \torus^\beta, 
        \ z = (z_\alpha,z_\beta);
        \quad
        \mathbb{Z}^{1+d} = \mathbb{Z}^\alpha \times \mathbb{Z}^\beta, 
        \ \zeta = (\zeta_\alpha,\zeta_\beta)
    \end{equation*}
    Recalling the definition of $\Sigma$ in \eqref{eq::def-characteristic-set}, for all $\zeta_\beta \in \mathbb{Z}^\beta$, we let
    \begin{equation*}
        \Sigma(\zeta_\beta) = \{\zeta_\alpha \in \mathbb{Z}^\alpha : (\zeta_\alpha,\zeta_\beta) \in \Sigma\}.
    \end{equation*}

    Let $u \in \mathscr{S}$.
    Since $\supp \widehat{u} \subset \Sigma$, we rearrange its Fourier series and write
    \begin{equation*}
        u(z_\alpha,z_\beta) 
        \propto_d \sum_{\zeta_\beta \in \mathbb{Z}^\beta} e^{i \zeta_\beta \cdot z_\beta} F_{\zeta_\beta}(z_\alpha),
        \quad
        F_{\zeta_\beta}(z_\alpha) 
        = \sum_{\zeta_\alpha \in \Sigma(\zeta_\beta)} e^{i \zeta_\alpha \cdot z_\alpha} \widehat{u}(\zeta_\alpha,\zeta_\beta).
    \end{equation*}
    We claim that 
    $\| F_{\zeta_\beta}\|_{L^p_{z_\alpha}}
    \lesssim \|\widehat{u}(\cdot,\zeta_\beta)\|_{\ell^2(\Sigma(\zeta_\beta))}$.
    Indeed, if $(\sharp \alpha,p) = (1,\infty)$, then this is trivial since $\sharp \Sigma(\zeta_\beta) \le 2$.
    If $(\sharp \alpha,p) = (2,4)$, then this follows from Zygmund's $L^4$ estimate.
    Consequently, by Parseval's theorem and the Minkowski inequality:
    \begin{align*}
        \|u\|_{L^p_{z_\alpha} L^2_{z_\beta}}^2
        & \propto_d \|F_{\zeta_\beta}(z_\alpha)\|_{L^p_{z_\alpha} \ell^2_{\xi_\beta}}^2
        \le \|F_{\zeta_\beta}(z_\alpha)\|_{\ell^2_{\xi_\beta} L^p_{z_\alpha}}^2 \\
        & \lesssim \sum_{\zeta_\beta \in \mathbb{Z}^\beta} \sum_{\zeta_\alpha \in \Sigma(\zeta_\beta)} |\widehat{u}(\zeta_\alpha,\zeta_\beta)|^2
        = \sum_{\zeta \in \mathbb{Z}^{d+1}} |\widehat{u}(\zeta)|^2
        \propto_d \|u\|_{L^2_z}^2.
        \qedhere
    \end{align*}
\end{proof}

\subsection{Lower Dimensional Examples}

In lower dimensions, we have the following examples of subspaces of $Y^{b,b'}$ where $b,b'$ may be contained in the interval $(-1/2,1/2)$.

\begin{lemma}
    If $d=1$, then $L^2_{t,x} \hookrightarrow Y^{3/8,-3/8}_0$.
\end{lemma}
\begin{proof}
    By Bourgain \cite{Bourgain1993fourier-schrodinger} and duality, one has $X^{3/8} \hookrightarrow L^4_{t,x}$ and $L^{4/3}_{t,x} \hookrightarrow X^{-3/8}$.
    Therefore, by Hölder's inequality, for all $\phi \in L^2_{t,x}$, there holds
    \begin{equation*}
        \|\phi\|_{Y^{3/8,-3/8}}
        \lesssim \|\phi\|_{\mathfrak{L}(L^4_{t,x}, L^{4/3}_{t,x})}
        \lesssim \|\phi\|_{L^2_{t,x}}.
        \qedhere
    \end{equation*}
\end{proof}

\begin{corollary}
    When $d=1$, the Schrödinger propagator $\schrodinger_V$ is $\chi$-observable if
    \begin{equation*}
        V \in L^2_x,
        \quad
        \chi \in L^2_{t,x} \setminus \{0\}.
    \end{equation*}
\end{corollary}

\begin{lemma}
    If $d=2$, then for all $b > 1/2$ and $\theta, \theta' \in [0,1]$, there holds
    \begin{equation*}
        L^{4/(\theta+\theta')}_x L^\infty_t \hookrightarrow Y^{b\theta,-b\theta'}.
    \end{equation*}
    Consequently, for all $p > 2$, there exists $b > 1/2$ and $\epsilon > 0$ such that
    \begin{equation*}
        L^p_x \hookrightarrow Y^{b,b-1+\epsilon}_0.
    \end{equation*}
\end{lemma}

\begin{proof}
    Since $b > 1/2$, interpolating between the embedding $X^b \hookrightarrow L^4_x L^2_t$ (by \cite{BourgainBurqZworski2013rough-potential}) and the identity $X^0 = L^2_{t,x}$, and then using duality, one obtains
    \begin{equation*}
        X^{b\theta} \hookrightarrow L^{4/(2-\theta)}_x L^2_t,
        \quad
        L^{4/(2+\theta)}_x  L^2_t \hookrightarrow X^{-b\theta}.
    \end{equation*}
    By Hölder's inequality, for all $\phi \in L^{4/(\theta+\theta')}_x L^\infty_t$, there holds
    \begin{equation*}
        \|\phi\|_{Y^{b\theta,-b\theta'}}
        \lesssim \|\phi\|_{\mathfrak{L}(L^{4/(2-\theta)}_x L^2_t, L^{4/(2+\theta')}_x L^2_t)}
        \lesssim \|\phi\|_{L^{4/(\theta+\theta')}_x L^\infty_t}.
        \qedhere
    \end{equation*}
\end{proof}

\begin{corollary}
    When $d=2$, the Schrödinger propagator $\schrodinger_V$ is $\chi$-observable if
    \begin{equation*}
        V \in L^{2+}_x,
        \quad
        \chi \in L^\infty_t \otimes L^4_x \setminus \{0\}.
    \end{equation*}
\end{corollary}

\subsection{Weak Topologies and Weak Limits}

We now prove Theorem~\ref{thm::equivalence-UI-schrodinger}.
Since the proofs for both sets of statements are essentially the same, we shall only consider the first half of the theorem.
Although the arguments in this section extend to general manifolds, we restrict our discussion to tori for simplicity.

Recall that, a bounded subset $S \subset L^2_{t,x}$ is uniformly square integrable if the family $|S|^2 = \{|f|^2 : f \in S\}$ is uniformly integrable.
Precisely, this means, for all $\epsilon > 0$, there exists $\delta > 0$ such that, for all $f \in S$ and for any $\omega \subset \torus_t \times \torus^d_x$ with $|\omega| < \delta$, there holds
\begin{equation*}
    \||f|^2\|_{L^1(\omega)} = \|f\|_{L^2(\omega)}^2 < \epsilon.
\end{equation*}

\begin{proof}[Proof of Theorem~\ref{thm::equivalence-UI-schrodinger}]
As already mentioned below Theorem~\ref{thm::equivalence-UI-schrodinger}, it remains to show that $\mathscr{S}_*$ is uniformly square integrable if and only if $Y_0 \supset L^\infty_{t,x}$.

First, by a density argument and the definition of the space $Y_0$, if $Y_0 \supset L^\infty_{t,x}$, then any weak-$*$ limit of $|\mathscr{S}_*|^2$ is also its weak limit.
By the Dunford--Pettis theorem, this implies the relative compactness of $K_*$ with respect to the weak topology $\sigma(L^1,L^\infty)$ and thus the uniform integrability of $\mathscr{S}_*$.

Conversely, by the De La Vallée--Poussin theorem, there exists a strict Young function $\Psi$ such that $\|\Psi(|u|^2)\|_{L^1_{t,x}} \le M$ for some finite $M > 0$ and for all $u \in \mathscr{S}_*$.
Let $\chi \in L^\infty_{t,x}$ and let $(\chi_j)_{j \ge 0}$ be a bounded sequence in $C^0_{t,x}$ such that $\chi_j \to \chi$ in $L^2_{t,x}$ as $j \to \infty$.
We claim that $\chi^n \to \chi$ in $Y$ which would conclude the proof.

To show  this, for $R > 0$, let $\delta(R)$ be the supremum of $s/\Psi(s)$ among all $s \ge R$.
The hypothesis on $\Psi$ implies that $\delta(R) = o(1)_{R \to \infty}$.
For all $u \in \mathscr{S}_*$, we have
    \begin{align*}
        \|(\chi_j - \chi) u\|_{L^2_{t,x}}^2
        & = \||\chi_j-\chi|^2 |u|^2 \bm{1}_{|u|^2 \ge R} \|_{L^1_{t,x}} + \||\chi_j-\chi|^2 |u|^2 \bm{1}_{|u|^2 \le R} \|_{L^1_{t,x}} \\
        & \le M \delta(R) \|\chi_j - \chi\|_{L^\infty_{t,x}}^2 + R \|\chi_j - \chi\|_{L^2_{t,x}}^2.
    \end{align*}
    For all $\epsilon > 0$, first fix $R > 0$ such that the first term is $< \epsilon/2$ and choose $N > 0$ such that the second term is also $< \epsilon/2$ when $j > N$, we obtain $\|(\chi_j - \chi) u\|_{L^2_{t,x}}^2 < \epsilon$.
\end{proof}

\subsection{Riemann Integrable Scenario}
\label{sec::riemann}

We now show that $Y_0$ contains all Riemann integrable functions, finishing the proof of Theorem~\ref{thm::Y_0-algebra-subset}.
We need the following approximation lemma.

\begin{lemma}
    If $\chi$ is a Riemann integrable function on $\torus_t \times \torus^d_x$, then there exists a sequence $(\chi_j)_{j \ge 0}$ in $C^0_{t,x}$ and a sequence $(Q_j)_{j \ge 0}$ of compact subsets of $\torus_t \times \torus^d_x$ which decreases to a zero measure set, such that
    \begin{equation*}
        \lim_{j \to \infty} \|\chi_j - \chi\|_{L^\infty_{t,x}(Q_j^c)} = 0.
    \end{equation*}
\end{lemma}
\begin{proof}
    Let $F_{\delta,\epsilon} = \{x \in M : \omega_\delta(x) < \epsilon\}$ for sufficiently small $\delta > 0$ and $\epsilon > 0$, where $\omega_\delta(x)$ denotes the supremum of $|\chi(y) - \chi(x)|$ when $\dist(x,y) < \delta$.
    There exists an open set $G_{\delta,\epsilon}$ such that $F_{4\delta,\epsilon} \subset G_{\delta,\epsilon} \subset F_{\delta,2 \epsilon}$ and $\dist(G_{\delta,\epsilon},F_{\delta,2\epsilon}^c) \ge \delta$.
    Indeed, let $B_x(\delta)$ be the (geodesic) ball of radius $\delta$ centered at $x$.
    Then it suffices to put
    \begin{equation*}
        G_{\delta,\epsilon} = \bigcup_{x \in F_{4\delta,\epsilon}} B_x(\delta).
    \end{equation*}
    Since $\chi$ is Riemann integrable, for all $\epsilon > 0$, there holds $|F_{\delta,\epsilon}^c| \to 0$ as $\delta \to 0$.
    Hence, for $j \ge 0$, there exists $\delta_j > 0$ such that $|F_{4\delta_j,2^{-j}}^c| < 2^{-j}$.
    Let
    \begin{equation*}
        Q_j = \bigcap_{1 \le \ell \le j} G_{\delta_\ell,2^{-\ell}}^c \subset G_{\delta_j,2^{-j}}^c,
    \end{equation*}
    then $(Q_j)_{j \ge 0}$ is a decreasing sequence of compact subsets satisfying $|Q_j| < 2^{-j}$.

    Let $\rho$ be a smooth density (so that it is nonnegative and its integral equals to one) on $\mathbb{R} \times \mathbb{R}^d$ which is in addition supported inside the unit ball.
    Let $\lambda_j > 1/\delta_j$ and define
    \begin{equation*}
        \chi_j(t,x) = \lambda_j^{1+d} \int_{\mathbb{R}} \int_{\mathbb{R}^d} \rho(\lambda_j s,\lambda_j v) \chi \bigl( t-s, x-v \bigr) \diff v \diff s.
    \end{equation*}
    Then clearly $(\chi_j)_{j \ge 0}$ is a bounded sequence in $C^\infty_{t,x}$ which satisfies $\|\chi_j\|_{L^\infty_{t,x}} \le \|\chi\|_{L^\infty_{t,x}}$ and $|\chi_j(x) - \chi(x)| \le 2^{1-j}$ for all $x \in G_{\delta_j,2^{-j}} \supset Q_j^c$.
\end{proof}
    
Now we proceed with the proof.
Let $\chi$ be a Riemann integrable function on $\torus_t \times \torus^d_x$.
Choose $(\chi_j)_{j \ge 0}$ and $(Q_j)_{j \ge 0}$ as in the above lemma.
We claim that $\chi_j \to \chi$ in $Y$ and thus $\chi \in Y_0$.
Indeed, if this is untrue, then there exists $\delta > 0$ a sequence $(u_j)_{j \ge 0}$ in $\mathscr{S}_*$ such that, up to the extraction of subsequence, there holds
\begin{align*}
   \delta 
   \le \|(\chi_j - \chi) u_j\|_{L^2_{t,x}}^2
   &  = \||\chi_j - \chi|^2 |u_j|^2 \bm{1}_{Q_j^c}\|_{L^1_{t,x}} + \||\chi_j-\chi|^2 |u_j|^2 \bm{1}_{Q_j}\|_{L^1_{t,x}}  \\
   & \le \|\chi_j - \chi\|_{L^\infty_{t,x}(Q_j^c)}^2 \|u_j\|_{L^2_{t,x}}^2 + \|\chi_j - \chi\|_{L^\infty_{t,x}}^2 \|u_j\|_{L^2_{t,x}(Q_j)}^2.
\end{align*}
Taking the limit $j \to \infty$, one obtains that, for some $C > 0$, 
\begin{equation*}
    0 < \delta/C 
    \le \liminf_{j \to \infty} \|u_j\|_{L^2_{t,x}(Q_j)}^2.
\end{equation*}

Since $Q_j$ is compact, there exists $\phi_j \in C^0_{t,x}$ such that $\bm{1}_{Q_j} \le \phi_j \le 1$.
One may further assume that $\supp \phi_j$ decreases to a zero measure set as $j \to \infty$.
We further assume that $\phi_j(x)$ is decreasing as $j \to \infty$ for each fixed $x$.
By Bourgain \cite{Bourgain1997lattice}, we may also assume that $|u_j|^2$ converges with respect to the weak-$*$ topology to some $f \in L^1_{t,x}$.

We will now hit a contradiction and conclude the proof.
Indeed, 
\begin{align*}
    0 < \liminf_{j \to \infty} \|u_j\|_{L^2_{t,x}(Q_j)}^2
    & \le \liminf_{j \to \infty} \|\phi_j |u_j|^2\|_{L^1_{t,x}} \\
    & \le \liminf_{\ell \to \infty} \liminf_{j \to \infty} \|\phi_\ell |u_j|^2\|_{L^1_{t,x}}
    = \liminf_{\ell \to \infty} \|\phi_\ell f\|_{L^1} = 0.
\end{align*}

\section{Observability for the Free Propagator}
\label{sec::proof-free-propagator}

In this section, we prove Theorem~\ref{thm::main-free}.
We assume that $\chi \in Y^{\infty,-\infty}_{0,+} \setminus \{0\}$ without further specifications.
First, let us discuss two necessary ingredients described in \S\ref{sec::strategy-of-proof}.

\subsection{Galilean Invariance}
\label{sec::galilean-V=0}

For $p \in \mathbb{R}^d$, define the \emph{Galilean transform} $\mathcal{G}_p$ by
\begin{equation*}
    (\mathcal{G}_p u)(t,x) = e^{i(p\cdot x - |p|^2 t)} u(t,x - 2pt).
\end{equation*}
If $p \in \mathbb{Z}^d$, then $\mathcal{G}_p$ preserves the space-time periodicity and defines an isometry on  $L^2_{t,x}$.
It also restricts to an isometry on $\mathscr{S}$.
In fact, if $u(t) = e^{it\Delta} u_0$, then
\begin{equation}
    \label{eq::galilean-free-propagator-relation}
    e^{it\Delta} (e^{ip\cdot x} u_0(x)) = (\galilean_p u)(t,x).
\end{equation}

For $S \subset \mathbb{Z}^{1+d}$, let $\Pi_S$ be the Fourier projector onto the frequency set $S$.
If $\bm{1}_S$ depends only on $n$ resp.\ $k$ (or simply if $S$ is a subset of $\mathbb{Z}$ resp.\ $\mathbb{Z}^d$), then $\Pi_S$ is identified with a temporal resp.\ spatial Fourier projector.

For $S \subset \mathbb{Z}^{1+d}$, Let $L^2_S = \Pi_S(L^2_{t,x})$.
If $p \in \mathbb{Z}^d$, then the multiplier $e^{i p \cdot x}$ is an isometry between $L^2_S$ and $L^2_{-p + S}$.
Let $\Lambda$ be an sublattice of $\mathbb{Z}^d$.
For any affine sublattice $\Gamma \in \affine_\Lambda \coloneqq \{q + \Lambda : q \in \mathbb{Z}^d\}$ and any $p \in \Lambda^\perp \coloneqq \{k \in \mathbb{Z}^d : k \perp \Lambda\}$, the observation $\|\chi u\|_{L^2_{t,x}}$ is invariant by $\mathcal{G}_p$ provided that $u \in L^2_{\mathbb{Z} \times \Gamma}$.

\begin{lemma}
\label{lem::observation-integral-invariance}
    If $\Gamma \in \affine_\Lambda$, $p \in \Lambda^\perp$, and $u \in L^2_{\mathbb{Z} \times \Gamma}$, then
    \begin{equation*}
        \|\chi u\|_{L^2_{t,x}} = \|\chi \mathcal{G}_p u\|_{L^2_{t,x}}.
    \end{equation*}
\end{lemma}

\begin{proof}
    Since $p \in \Lambda^\perp$, there exists $\theta \in \mathbb{R}$ such that $p \cdot k = \theta$ for all $k \in \Gamma$.
    The following identity implies $|\mathcal{G}_p u| = |u|$ and concludes the lemma:
    \begin{equation*}
        e^{-i (p\cdot x - |p|^2 t)} \mathcal{G}_p u(t,x) 
        \propto_d \sum_{n \in \mathbb{Z}} \sum_{k \in \Gamma} e^{i(nt + k\cdot (x-2pt))} \widehat{u}(n,k)=  e^{-2it\theta} u(t,x).
        \qedhere
    \end{equation*}
\end{proof}

Hence, if the observability holds for all $u_0 \in L^2_\Gamma$, then it holds for all $u_0 \in L^2_{-p + \Gamma}$ when $p \in \Lambda^\perp$.
This leads to the consideration of the group action:
\begin{equation}
    \Lambda^\perp \curvearrowright \affine_\Lambda : (p,\Gamma) \to p + \Gamma.
\end{equation}
This group action has a finite number of orbits.
In fact, the following lemma gives the exact number of orbits when $\Lambda$ is \emph{primitive}, i.e., $\lsp(\Lambda) \cap \mathbb{Z}^d = \Lambda$ where $\lsp(\Lambda)$ is the linear span of $\Lambda$ in $\mathbb{R}^d$.

\begin{lemma}
    \label{lem::sublattice-vertical-translation-orbit}
    If $\Lambda$ is primitive, then $\sharp(\affine_\Lambda / \Lambda^\perp) = \det(\Lambda)^2$.
\end{lemma}

\begin{proof}
    Let $\tilde{\Lambda} = \proj_{\Lambda^\perp} (\mathbb{Z}^d)$.
    Notice that, for every $\Gamma = p + \Lambda \in \affine_\Lambda$ and for every $k \in \Gamma$, there holds $\proj_{\Lambda^\perp} k = \proj_{\Lambda^\perp} p$.
    Therefore, we have a bijection
    \begin{equation*}
        \affine_\Lambda \to \tilde{\Lambda}, \quad p + \Lambda \mapsto \proj_{\Lambda^\perp} p.
    \end{equation*}
    Since $\Lambda^\perp$ is a subgroup of $\tilde{\Lambda}$, it naturally acts on $\tilde{\Lambda}$ by addition.
    The above mapping establishes an isomorphism of group actions: $\Lambda^\perp \curvearrowright \affine_\Lambda \simeq \Lambda^\perp \curvearrowright \tilde{\Lambda}$.
    Therefore, these two group actions have equal numbers of orbits:
    \begin{equation*}
        \sharp(\affine_\Lambda / \Lambda^\perp)
        = \sharp(\tilde{\Lambda} / \Lambda^\perp)
        = \det(\Lambda^\perp) / \det(\tilde{\Lambda}).
    \end{equation*}
    To conclude, we use the fact that $\Lambda$ is primitive, which implies (see \cite{McMullen1984lattice}):
    \begin{equation*}
        \det(\Lambda^\perp) = \det(\Lambda),
        \quad
        \det(\tilde{\Lambda}) = \det(\Lambda)^{-1}.
        \qedhere
    \end{equation*}
\end{proof}

\subsection{Cluster Structure}
\label{sec::cluster-decomposition}

On the lattice set $\Sigma$ we have the following cluster structure
(see e.g., \cite{Bourgain1997lattice,Bourgain1998quasi-periodic,Bourgain2002survey,Bourgain2005lattice}; see also \cite{BertiMaspero2019longtime}).

\begin{lemma}
    \label{lem::GS}
    Let $d \ge 1$ and let $S \subset \Sigma$.
    For $R \ge 1$, connecting every pair of distinct points in $S$ with distance $\le 100 R$ gives an undirected graph.
    Then, there exists $K_d > 0$ which depends solely on $d$ such that, the diameter of every connected component (which is called a cluster) of this graph is $\le R^{K_d}$.
\end{lemma}

Let $\Gamma$ be an affine sublattice and apply the above lemma to the subset
\begin{equation*}
    \Sigma(\Gamma) = \Sigma \cap (\mathbb{Z} \times \Gamma).
\end{equation*}
For $R \ge 1$, let $Q^\alpha_\Gamma$ be the connected components of the graph indexed by $\alpha \in \mathscr{I}(\Gamma;R)$, and let $Z^\alpha_\Gamma \subset \Gamma$ be the projection of $Q^\alpha_\Gamma$ on $\Gamma$ via $(n,k) \mapsto k$.
We have the disjoint unions
\begin{equation*}
    \Sigma(\Gamma) = \bigcup_{\alpha \in \mathscr{I}(\Gamma;R)} Q^\alpha_\Gamma,
    \quad
    \Gamma = \bigcup_{\alpha \in \mathscr{I}(\Gamma;R)} Z^\alpha_\Gamma.
\end{equation*}
Also notice that, from the construction, if $\alpha \ne \beta$, then
\begin{equation}
    \label{eq::cluster-far-away}
    \dist(Q^\alpha_\Gamma,Q^\beta_\Gamma) > 100R.
\end{equation}

Let $\Gamma^\alpha$ be the smallest primitive affine sublattice containing $Z^\alpha_\Gamma$.
When $\Gamma = \mathbb{Z}^d$, Bourgain \cite{Bourgain2014control-3D} claimed that $\dim \Gamma^\alpha < \dim \Gamma$ ($= d$).
This holds only if $Z^\alpha$ is sufficiently far away from the origin, as implicitly implied by the proof of e.g., \cite[Lem.~9.36]{Bourgain1998quasi-periodic}.
In the following, we give another elementary proof that works for general affine sublattices.

Let $\Gamma \in \mathbb{A}_\Lambda$.
We prove $\dim \Gamma^\alpha < \dim \Gamma$ when the image of $Z^\alpha_\Gamma$ by $\proj_\Lambda$ (the orthogonal projection on $\lsp (\Lambda)$) is far from the origin.
Let $\mathscr{I}_\flat(\Gamma;R)$ resp.\ $\mathscr{I}_\sharp(\Gamma;R)$ be the set of all $\alpha \in \mathscr{I}(\Gamma;R)$ such that $\dim \Gamma^\alpha = \dim \Gamma$ resp.\ $\dim \Gamma^\alpha < \dim \Gamma$.
Then
\begin{equation*}
    \mathscr{I}(\Gamma;R) = \mathscr{I}_\flat(\Gamma;R) \cup \mathscr{I}_\sharp(\Gamma;R).
\end{equation*}

\begin{lemma}
    \label{lem::lattice-structure}
    Assume that $\dim \Gamma \ge 1$.
    There exists $M_d > 0$ depending solely on $d$ such that, for all $\alpha \in \mathscr{I}(\Gamma;R)$, it belongs to $\mathscr{I}_\sharp(\Gamma;R)$ provided that
    \begin{equation*}
        \sup_{k \in Z^\alpha_\Gamma} |\proj_\Lambda k| \ge M_d R^{(\dim \Gamma + 2)K_d}.
    \end{equation*}
\end{lemma}

\begin{proof}
    Assume that $\sharp Z^\alpha_\Gamma \ge 2$.
    Denote $\xi_\Lambda = \proj_\Lambda \xi$ for $\xi \in \mathbb{R}^d$.
    By the construction of clusters given in Lemma~\ref{lem::GS}, if $k,\ell \in Z^\alpha_\Gamma$ and $\xi = k-\ell$, then
    \begin{equation*}
        2|k \cdot \xi| 
        \le ||k|^2 - |\ell|^2| + |\xi|^2
        \le R^{K_d} + (R^{K_d})^2 \le 2 R^{2 K_d}.
    \end{equation*}
    Let $\theta(p,q) \in [0,\pi]$ be such that $p \cdot q = |p| \cdot |q| \cdot \cos \theta(p,q)$, then
    \begin{equation*}
        |k_\Lambda| \cdot |\cos \theta(k_\Lambda,\xi)|
        \le |k_\Lambda| \cdot |\xi| \cdot |\cos \theta(k_\Lambda,\xi)|
        = |k_\Lambda \cdot \xi| = |k \cdot \xi| \le R^{2K_d}.
    \end{equation*}
    Therefore $Z^\alpha_\Gamma \subset k + S_k$ for all $k \in Z^\alpha_\Gamma$, where
    \begin{equation*}
        S_k = \bigl\{ \xi \in \Lambda : |\xi| \le R^{K_d}, |\cos \theta(k_\Lambda,\xi)| \le R^{2K_d} / |k_\Lambda| \bigr\}.
    \end{equation*}
    Let $s = \dim \Gamma$.
    It remains to prove that, if $|\proj_\Lambda k| \ge M_d R^{(s + 2)K_d}$ for some $k \in Z^\alpha_\Gamma$ and some sufficiently large $M_d > 0$, then $\dim \lsp(S_k) < s$.

    We use proof by contradiction and assume that $\dim \lsp(S_k) = s$.
    Let $Q$ be the convex hull of $S_k$.
    Then $\mathcal{H}^s(Q) \gtrsim_d 1$ since $S_k \subset \mathbb{Z}^d$, where $\mathcal{H}^s$ is the $s$-dimensional Hausdorff measure.
    This gives a contradiction when $M_d$ is sufficiently large:
    \begin{equation*}
        \mathcal{H}^s(Q) 
        \le \frac{(R^{K_d})^s \times (R^{2 K_d})}{|k_\Lambda|} 
        = \frac{(R^{K_d})^{s + 2}}{|k_\Lambda|}
        \le M_d^{-1}.
        \qedhere
    \end{equation*}
\end{proof}

\subsection{Mathematical Induction}

We prove Theorem~\ref{thm::main-free} by inducting on the dimension of the Fourier support of $u_0$.
Precisely, for any $S \subset \mathbb{Z}^d$, we denote by $\dim S$ the dimension of the smallest affine sublattice in $\mathbb{Z}^d$ that contains $S$.

\begin{lemma}
    \label{lem::zero-dim-observability-free}
    If $\chi \in L^2_{t,x}$ and if $\dim \supp \widehat{u}_0 = 0$, then
    \begin{equation*}
        \|\chi e^{it\Delta} u_0\|_{L^2_{t,x}}
        = \|\chi\|_{L^2_{t,x}} \|u_0\|_{L^2_x}.
    \end{equation*}
\end{lemma}
\begin{proof}
    There exists $k \in \mathbb{Z}^d$ and $c_k \in \mathbb{C}$ such that $e^{it\Delta} u_0(x) = c_k e^{i (k\cdot x + |k|^2t)}$.
\end{proof}

For functions $f_\alpha$ indexed by $\alpha \in \mathscr{I}(\Gamma;R)$, where $\Gamma$ is an affine sublattice and $R \ge 1$, we write $f_* = (f_\alpha)_{\alpha \in \mathscr{I}(\Gamma;R)}$.
We also write, for any Banach space $E$:
\begin{equation*}
    \|f_*\|_{\ell^2 E} = \Bigl( \sum_{\alpha \in \mathscr{I}(\Gamma;R)} \|f_\alpha\|_E^2 \Bigr)^{1/2}.
\end{equation*}

\begin{proposition}
    \label{prop::weak-observability-V=0}
    Assume that the observability \eqref{eq::observability-generalized} holds for all initial data with Fourier supports of dimensions $\le \nu \le d-1$.
    If $\dim \Gamma \le \nu + 1$, then there exists a finite subset $B_\Gamma \subset \Gamma$ such that, for all $u_0 \in L^2_\Gamma$, there holds
    \begin{equation}
        \label{eq::weak-observability-V=0}
        \|u_0\|_{L^2_x}^2
        \lesssim_{d,\chi} \|\chi e^{it\Delta} u_0\|_{L^2_{t,x}}^2
        + \|\Pi_{B_\Gamma} u_0\|_{L^2_x}^2.
    \end{equation}
\end{proposition}

\begin{proof}
    Take the cluster decomposition of $\Sigma(\Gamma)$ with scale $R \ge 1$.
    For each $\alpha \in \mathscr{I}(\Gamma;R)$, let $u_\alpha = \Pi_{Q^\alpha_\Gamma} u$ and $u_{0,\alpha} = \Pi_{Z^\alpha_\Gamma} u_0$, then $u_\alpha(t) = e^{it\Delta} u_{0,\alpha}$.

    Let $\zeta \in \trig_{t,x}$ with $\deg \zeta \le R$ such that $\||\chi|^2 -|\zeta|^2\|_{Y^{\infty,-\infty}} = o(1)_{R \to \infty}$.
    By \eqref{eq::cluster-far-away} and orthogonality, one has $\|\zeta u\|_{L^2_{t,x}} = \|\zeta u_*\|_{\ell^2 L^2_{t,x}}$.
    Therefore,
    \begin{align*}
        \bigl| \|\chi u \|_{L^2_{t,x}}^2
        - \|\chi u_*\|_{\ell^2 L^2_{t,x}}^2 \bigr|
        & \le \bigl| \|\chi u\|_{L^2_{t,x}}^2 - \|\zeta u\|_{L^2_{t,x}}^2 \bigr|
        + \bigl| \|\chi u_*\|_{\ell^2 L^2_{t,x}}^2 - \|\zeta u_*\|_{\ell^2 L^2_{t,x}}^2 \bigr|\\
        & \le o(1)_{R \to \infty} \bigl( \|u_0\|_{L^2_x}^2 +  \|u_{0,*}\|_{\ell^2 L^2_x}^2 \bigr)
        \lesssim o(1)_{R \to \infty} \|u_0\|_{L^2_x}^2.
    \end{align*}
    
    If $\alpha \in \mathscr{I}_\sharp(\Gamma;R)$, then $\dim \Gamma^\alpha \le \nu$.
    Therefore $\|u_{0,\alpha}\|_{L^2_x} \lesssim_{d,\chi} \|\chi u_\alpha\|_{L^2_{t,x}}$ due to the induction hypothesis.
    Summing this up over $\alpha \in \mathscr{I}_\sharp(\Gamma;R)$, one obtains
    \begin{equation*}
        \|u_0\|_{L^2_x}^2
        \lesssim_{d,\chi} \|\chi u\|_{L^2_{t,x}}^2 + \|\Pi_{B_\Gamma} u_0\|_{L^2_x}^2
        + o(1)_{R \to \infty} \|u_0\|_{L^2_x}^2,
    \end{equation*}
    where $B_\Gamma$ is the union of all $Z^\alpha_\Gamma$ with $\alpha \in \mathscr{I}_\flat(\Gamma;R)$ and is finite according to Lemma~\ref{lem::lattice-structure}.
    We conclude by choosing $R$ sufficiently large.
\end{proof}

\subsection{Uniqueness-Compactness Argument}
\label{sec::uniqueness-compactness-V=0}

To eliminate the remainder term from the weak observability \eqref{eq::weak-observability-V=0} and obtain the (strong) observability, we use the uniqueness-compactness argument by Bardos, Lebeau and Rauch \cite{BardosLebeauRauch1992}.

A new challenge is to show the uniformity of the observability for all $\Gamma$ of dimension $\nu+1$.
To address this issue, we use a finiteness argument which goes as follows:

Notice that, when proving Proposition~\ref{prop::weak-observability-V=0}, we have chosen a large $R \ge 1$ independent of $\Gamma$.
Moreover, if $\mathscr{I}_\flat(\Gamma;R) = \emptyset$, then $B_\Gamma = \emptyset$ and \eqref{eq::weak-observability-V=0} gives the observability.
It remains to consider those $\Gamma$ with $\mathscr{I}_\flat(\Gamma;R) \ne \emptyset$.
Evidently, we may further restrict ourselves to those $\Gamma$ which are primitive.

Let $\affine_R$ be the set of all primitive affine sublattices $\Gamma$ with $\mathscr{I}_\flat(\Gamma;R) \ne \emptyset$.
On $\affine_R$ define the equivalence relation: $\Gamma_1 \sim \Gamma_2$ if there exists a primitive sublattice $\Lambda$ such that $\Gamma_1,\Gamma_2 \in \affine_\Lambda$ and $\Gamma_1 = p + \Gamma_2$ for some $p \in \Lambda^\perp$.
By Lemma~\ref{lem::observation-integral-invariance}, the observability estimates for all $\Gamma$ in the same equivalent class are equivalent.

\begin{lemma}
    \label{lem::A_R-finite}
    For all $d \ge 1$ and $R \ge 1$, there holds $\sharp(\affine_R / \sim) \le R^{2 d^2 K_d}$.
\end{lemma}

\begin{proof}
    Let $\lattice_R$ be the set of all primitive $\Lambda$ such that $\affine_\Lambda \cap \affine_R \ne \emptyset$.
    If $\Lambda \in \lattice_R$, $\Gamma \in \affine_\Lambda \cap \affine_R \ne \emptyset$ and $\alpha \in \mathscr{I}_\flat \ne \emptyset$, then $\dim \Gamma^\alpha = \dim \Gamma$.
    
    Since both $\Gamma$ and $\Gamma^\alpha$ are primitive, one has $\Gamma^\alpha = \Gamma$.
    Hence $\Gamma$ is the smallest affine sublattice containing $Z^\alpha_\Gamma$.
    By Lemma~\ref{lem::GS}, $\diam Z^\alpha_\Gamma \le R^{K_d}$.
    
    Therefore $\Lambda$ is spanned by some $S \subset \mathbb{Z}^d$ with $\diam S \le R^{K_d}$.
    This implies $\sharp \lattice_R \le R^{d^2 K_d}$ and $|\det(\Lambda)| \le R^{dK_d}$ for all $\Lambda \in \lattice_R$.
    Consequently, by Lemma~\ref{lem::sublattice-vertical-translation-orbit},
    \begin{equation*}
        \sharp(\affine_R / \sim)
        \le \sharp \lattice_R \times \sup_{\Lambda \in \lattice_R} \sharp(\affine_\Lambda/\Lambda^\perp)
        \le R^{d^2 K_d} \sup_{\Lambda \in \lattice_R} \det(\Lambda)^2
        \le R^{2d^2 K_d}.
        \qedhere
    \end{equation*}
\end{proof}

Now we proceed by contradiction.
This yields the existence of a sequence of primitive $\Gamma_j$ with $\dim \Gamma_j = \nu + 1$ and a sequence $u_{0,j} \in L^2_{\Gamma_j}$ such that
\begin{equation*}
    \|u_{0,j}\|_{L^2_x} = 1,
     \qquad
     \lim_{j \to \infty} \|\chi u_j\|_{L^2_{t,x}} = 0,
\end{equation*}
where $u_j = \schrodinger_0 u_{0,j}$.
We may further assume that $u_{0,j} \rightharpoonup u_0$ weakly in $L^2_x$ and $u_j \rightharpoonup u$ weakly in $L^2_{t,x}$.
Clearly $u = \schrodinger_0 u_0$.
By Lemma~\ref{lem::semi-continuity} below (with $b = \infty$), we have $\|\chi u\|_{L^2_{t,x}} = 0$, and hence $u$ vanishes almost everywhere on $\supp \chi$.

\begin{lemma}
    \label{lem::semi-continuity}
    Suppose that $(u_j)_{j \ge 0}$ is bounded in $X^b$ for some $b \ge 0$ and $u_j \rightharpoonup u$ weakly in $L^2_{t,x}$.
    Then for all $\chi \in L^2_{t,x}$ such that $|\chi|^2 \in Y^{b,-b}_{0,+}$, there holds
    \begin{equation*}
        \|\chi u\|_{L^2_{t,x}} 
        \le \liminf_{j \to \infty} \|\chi u_j\|_{L^2_{t,x}}.
    \end{equation*}
\end{lemma}
\begin{proof}    
    If $\chi \in C^0_{t,x}$, then $\chi u_j \rightharpoonup \chi u$.
    The lemma follows directly from the semi-continuity of Hilbert norms (the $L^2_{t,x}$ norm in our setting).
    In the general case, let $\chi_n \in C^0_{t,x}$ be such that $|\chi_n|^2 \to |\chi|^2$ in $Y^{b,-b}$.
    Then, the lemma applies to all $\chi_n$ and one obtains
    \begin{align*}
        \|\chi u\|_{L^2_{t,x}}^2
        & = \lim_{n \to \infty} \|\chi_n u\|_{L^2_{t,x}}^2
        \le \liminf_{n \to \infty} \liminf_{j \to \infty} \|\chi_n u_j\|_{L^2_{t,x}}^2 \\
        & = \liminf_{n \to \infty} \liminf_{j \to \infty} \bigl( \|\chi u_j\|_{L^2_{t,x}}^2 + o(1)_{n \to \infty} \bigr)
        = \liminf_{j \to \infty} \|\chi u_j\|_{L^2_{t,x}}
        \qedhere
    \end{align*}
\end{proof}

By the pigeon hole principle and Lemma~\ref{lem::A_R-finite}, we may assume that, up to a subsequence, all $\Gamma_j$ are in the same equivalent class.
By Lemma~\ref{lem::observation-integral-invariance}, we may further assume all $\Gamma_j$ to be equal and thus write $\Gamma_j = \Gamma$ for all $j$ and some affine sublattice $\Gamma$.
By Proposition~\ref{prop::weak-observability-V=0}, there exists a finite $B_\Gamma \subset \Gamma$, such that
\begin{equation*}
    1 = \|u_{0,j}\|_{L^2_x}^2
    \lesssim_{d,\chi} \|\chi u_j\|_{L^2_{t,x}}^2
    + \|\Pi_{B_\Gamma} u_{0,j}\|_{L^2_x}^2.
\end{equation*}
Since the observation vanishes as $j \to \infty$, and since $\Pi_{B_\Gamma}$ is compact on $L^2_x$ (due to the finiteness of $B_\Gamma$) we obtain that
\begin{equation*}
    \|\Pi_{B_\Gamma} u_0\|_{L^2_x}
    = \lim_{j \to \infty} \|\Pi_{B_\Gamma} u_{0,j}\|_{L^2_x}
    \gtrsim_{d,\chi} 1.
\end{equation*}
Hence $u_0 \ne 0$ and thus $u \ne 0$.
This contradicts the fact that nonzero solutions to the free Schrödinger equation cannot vanish on a domain of positive measure.

\begin{lemma}
    \label{lem::unique-continuation-V=0}
    If $u \in \mathscr{S}$ and $|u^{-1}(0)| > 0$, then $u = 0$.
\end{lemma}

\begin{proof}
    Expressing $u$ in terms of its Fourier series, one sees that, for almost all $x \in \torus^d$, the function $u_x : t \mapsto u(t,x)$ extends to a holomorphic function in the Hardy space $H^2(\mathbb{C}^-)$ on the half plane $\mathbb{C}^- = \{a+bi \in \mathbb{C} : a \in \mathbb{R}, b < 0\}$.

    By Fatou's theorem, this extension converges non-tangentially almost everywhere to $u_x$.
    By the Lusin--Privalov uniqueness theorem \cite{LusinPrivalov1925} and \eqref{eq::time-section-positive-measure}, one deduces that $u_x = 0$ for almost all $x \in \omega$, where
    \begin{equation}
        \label{eq::time-section-positive-measure}
        \omega = \{x \in \torus^d : |u^{-1}(0) \cap (\torus \times \{x\})|  > 0 \}.
    \end{equation}
    By the Fubini theorem $|\omega| > 0$.      
    Consequently $u$ vanishes identically on $\mathbb{R} \times \omega$.

    Following the uniqueness-compactness argument (to be elaborated in \S\ref{sec::unique-compactness-argument} under the presence of the potential $V$), if there exists $u \in \mathscr{S} \setminus \{0\}$ which vanishes on $\mathbb{R} \times \omega$, then there exists an eigenfunction of $\hamiltonian_0 = -\Delta$ which vanishes on $\omega$.
    Clearly, since toral eigenfunctions are analytic, this is impossible.
\end{proof}

\section{Observability for General Propagators}
\label{sec::proof-general-propagator}

In this section, we prove Theorem~\ref{thm::main-general}.
Without further specifications, we assume that $V$ and $\chi$ satisfy the conditions stated in the theorem.
In fact, in the proof, we only need $V$ to satisfy the weaker condition \eqref{eq::cdt-V-weaker}.
Also recall that $\chi$ satisfies the support condition \eqref{eq::cdt-chi-supp} and is identified with its periodic extension to $\torus_t \times \torus^d_x$.

For any sublattice $\Lambda$, denote $\hamiltonian_\Lambda = \hamiltonian_{V_\Lambda}$ and $\schrodinger_\Lambda = \schrodinger_{V_\Lambda}$, where $V_\Lambda = \Pi_\Lambda V$.
We aim to prove the following theorem which clearly implies Theorem~\ref{thm::main-general}.

\begin{theorem}
    \label{thm::observability-Lambda-V-general}
    If $\Lambda$ is a primitive sublattice, $\Gamma \in \affine_\Lambda$ and $u_0\in L^2_\Gamma$, then
    \begin{equation}
        \label{eq::observability-Lambda-V-general}
        \|u_0\|_{L^2_x} \lesssim_{d,V,\chi} \|\chi \schrodinger_\Lambda u_0\|_{L^2(\mathbb{R}_t \times \torus^d_x)}.
    \end{equation}
\end{theorem}

\subsection{Cluster Structure}

When $V\ne 0$, an extension of the cluster decomposition to neighborhoods of $\Sigma$ is required.
Let $\Gamma$ be a primitive affine sublattice and take the custer decomposition with parameter $R \ge 1$ as in \S\ref{sec::cluster-decomposition}.
Then define
\begin{align*}
        N^\alpha_\Gamma & = \{(n,k) \in \mathbb{Z} \times Z^\alpha_\Gamma : n + |k|^2 \le 10 R\}, \quad \alpha \in \mathscr{I}(\Gamma;R);\\
        N^c_\Gamma & = \{(n,k) \in \mathbb{Z} \times \Gamma : n + |k|^2 > 10 R\}.
\end{align*}
Clearly $N^\alpha_\Gamma \supset Q^\alpha_\Gamma$, and we have the decomposition
\begin{equation*}
    \mathbb{Z} \times \Gamma = \Bigl( \bigcup_{\alpha \in \mathscr{I}(\Gamma;R)} N^\alpha_\Gamma \Bigr)\cup N^c_\Gamma.
\end{equation*}

If $(n,k) \in N^\alpha_\Gamma$, then we have
\begin{equation*}
    \dist((n,k),Q^\alpha_\Gamma) \le |(n,k)-(- |k|^2,k)| = |n + |k|^2| \le 10 R.
\end{equation*}
Combining this and \eqref{eq::cluster-far-away}, we deduce that, if $\alpha \ne \beta$, then
\begin{equation}
    \label{eq::cluster-far-away-N}
    \dist(N^\alpha_\Gamma,N^\beta_\Gamma) > 10 R.
\end{equation}
Let $X^b_\Gamma$ be the subspace of $X^b$ consisting of functions whose Fourier supports are contained in $\mathbb{Z} \times \Gamma$.
Then we have the following estimate.
\begin{lemma}
    \label{lem::est-far-from-paraboloid}
    If $f \in X^b_{\Gamma}$ and $\epsilon > 0$, then $\|\Pi_{N^c_\Gamma} f\|_{X^{b-\epsilon}} \lesssim R^{-\epsilon} \|f\|_{X^b}$.
\end{lemma}
\begin{proof}
    This follows directly from the definition of Bourgain spaces:
    \begin{align*}
        \|\Pi_{N^c_\Gamma} f\|_{X^{b-\epsilon}}
        & = \|\langle n + |k|^2 \rangle^{b-\epsilon} (\bm{1}_{N^c_\Gamma} \widehat{f})(n,k)\|_{\ell^2_{n,k}} \\
        & \lesssim R^{-\epsilon} \|\langle n + |k|^2 \rangle^{b} (\bm{1}_{N^c_\Gamma} \widehat{f})(n,k)\|_{\ell^2_{n,k}}
        \lesssim R^{-\epsilon} \|f\|_{X^b}.
        \qedhere
    \end{align*}
\end{proof}

\subsection{Periodized Schrödinger Equation}

We study a periodized Schrödinger equation whose solutions are, within small time, given by the Schrödinger propagator.

Precisely, for any initial data $u_0$ and any exterior force $g$, consider the inhomogeneous and periodized Schrödinger equation in the Duhamel form:
\begin{equation}
    \label{eq::schrodinger-duhamel-inhomogeneous}
    u = \schrodinger_0 u_0 + \mathcal{E} \eta \mathfrak{D} V_\Lambda u + g,
\end{equation}
where $\eta \in C_c^\infty(\mathbb{R})$ equals to $1$ near the origin, $\mathcal{E}$ is the periodic extension from $(-\pi,\pi)$ to $\torus$, and $\mathfrak{D}$ is the Duhamel operator defined by:
\begin{equation*}
    (\mathfrak{D} f)(t) 
    = -i\int_0^t e^{i(t-s)} f(s) \diff s.
\end{equation*}

Via a fixed point argument, we show that \eqref{eq::schrodinger-duhamel-inhomogeneous} is well-posed in Bourgain spaces.
The following classical estimate (see e.g., \cite{Ginibre1996Bourbaki}) is needed.

\begin{lemma}
    \label{lem::Duhamel-contraction}
    If $b \in [0,1]$, $\epsilon \in [0,1-b]$, and $b + \epsilon > 1/2$, then for $\tau \in (0,1]$,
    \begin{equation}
        \label{eq::Duhamel-contraction}
        \|\mathcal{E} \eta_\tau \mathfrak{D}\|_{X^{b-1+\epsilon} \to X^{b}}
        \lesssim_{d,b,\eta} \tau^\epsilon,
        \quad
        \eta_\tau(t) = \eta(t/\tau).
    \end{equation}
\end{lemma}

\begin{proposition}
    \label{prop::schrodinger-duhamel-well-posedness}
    If $b \in (1/2, 1)$ and $V \in Y^{b,b-1+\epsilon}$ for some $\epsilon > 0$, then there exists $\eta \in C_c^\infty(\mathbb{R})$ equaling to $1$ near the origin such that for all $u_0 \in L^2_x$ and all $g \in X^b$, the equation \eqref{eq::schrodinger-duhamel-inhomogeneous} admits a unique solution $u \in X^b$.
    Moreover $u$ satisfies
    \begin{equation}
        \label{eq::est-energy-periodized}
        \|u\|_{X^b} \lesssim_{d,b} \|u_0\|_{L^2_x} + \|g\|_{X^b}.
    \end{equation}
\end{proposition}

\begin{proof}
    By Lemma~\ref{lem::projected-norm-est} below, $V_\Lambda \in Y^{b,+}$.
    Choose any cutoff function $\eta$ and replace it, if necessary, with $\eta_\tau$ where $\tau$ is sufficiently small.
    By \eqref{eq::Duhamel-contraction}, the operator $\mathcal{E} \eta \mathfrak{D} V_\Lambda$ is a contraction on $X^b$.
    We conclude by the Banach fixed point theorem.
\end{proof}

\begin{lemma}
    \label{lem::projected-norm-est}
    If $\|\cdot\|_*$ is any seminorm that is translation invariant, i.e., for all $y \in \mathbb{R}^d$, there holds $\|V(\cdot + y) \|_* = \|V\|_*$, then $\|V_\Lambda\|_* \le \|V\|_*$.
\end{lemma}

\begin{proof}
    It suffices to take the $\|\cdot\|_*$ seminorm and apply the Minkowski inequality in the following identity where $\torus_{\Lambda^\perp} = \lsp(\Lambda^\perp) / \Lambda^\perp$:
    \begin{equation*}
        V_\Lambda(x) 
        = \fint_{\torus_{\Lambda^\perp}} V(x+y) \diff y.
        \qedhere
    \end{equation*}
\end{proof}
    
\subsection{Approximate Equations}

Fix $\eta \in C_c^\infty(\mathbb{R})$ as in Proposition~\ref{prop::schrodinger-duhamel-well-posedness}.
We now introduce an approximate equation to \eqref{eq::schrodinger-duhamel-inhomogeneous} and show its well-posedness in $X^b$.

\begin{lemma}
    \label{lem::duhamel-op-str-est}
    If $\supp \eta \subset (-\pi,+\pi)$, then
    \begin{equation}
        \label{eq::E-eta-L-expression}
        \mathcal{E} \eta \mathfrak{D} = \mathcal{E} \eta \Theta + \mathcal{E}\eta \schrodinger_0 \mathcal{T} \Theta + i\mathcal{E}t\eta\schrodinger_0 \mathcal{P},
    \end{equation}
    where $\Theta$ is the Fourier multiplier of symbol $\theta$ which equals to $(n+|k|^2)^{-1}$ when $n+|k|^2 \ne 0$ and vanishes otherwise,  whereas the operators $\mathcal{T}$ and $\mathcal{P}$ are defined by
    \begin{equation*}
        \widehat{\mathcal{T} f}(k) = \sum_{n \in \mathbb{N}} \widehat{f}(n,k),
        \quad
        \widehat{\mathcal{P} f}(k) = \widehat{f}(-|k|^2,k).
    \end{equation*}
    Moreover, $\mathcal{P} \in \mathfrak{L}(X^{-b},L^2_x)$ and $\Theta \in Y^{b-1,b}$ for $b \in \mathbb{R}$, while $\mathcal{T} \in \mathfrak{L}(X^b,L^2_x)$ for $b > 1/2$.
\end{lemma}

\begin{proof}
    By direct computation on the Fourier side:
    \begin{equation*}
            \widehat{\mathcal{E} \eta \mathfrak{D} f}(n,k)
            = \sum_{m \in \mathbb{Z}} \bigl( \widehat{\mathcal{E}\eta}(n-m) - \widehat{\mathcal{E}\eta}(n + |k|^2) \bigr) \theta(m,k) \widehat{f}(m,k)
            + i \widehat{\mathcal{E} t \eta}(\omega) \widehat{f}(-|k|^2,k).
    \end{equation*}
    The expression \eqref{eq::E-eta-L-expression} follows via identifying the terms with the corresponding operators.
    The estimates for $\Theta$ and $\mathcal{P}$ follows from the definition of Bourgain spaces.
    To estimate $\mathcal{T}$, we use Parseval's identity and the Cauchy-Schwartz inequality:
    \begin{equation*}
        \|\mathcal{T} f\|_{L^2_x}^2
        \le \sum_{k \in \mathbb{Z}^d} \biggl( \sum_{n\in \mathbb{Z}} \langle n + |k|^2 \rangle^{-2b} \times  \sum_{n\in \mathbb{Z}} \langle n + |k|^2 \rangle^{2b} |\widehat{f}(n,k)|^2 \biggr)
        \asymp_b \|f\|_{X^b}^2.
        \qedhere
    \end{equation*}
\end{proof}

Replacing $\mathcal{E}\eta$ and $\mathcal{E}(t\eta)$ with $\phi,\psi \in \trig_t$, we define the operator $\mathcal{K} = \mathcal{K}_{\phi,\psi}$ as follows (we always hide the dependence on $\phi,\psi$ unless otherwise specified):
\begin{equation*}
    \mathcal{K} = \phi \Theta + \phi \schrodinger_0 \mathcal{T} \Theta + i\psi\schrodinger_0 \mathcal{P}.
\end{equation*}
Let $W \in \trig_x$ and $W_\Lambda = \Pi_\Lambda W$.
The approximate equation of \eqref{eq::schrodinger-duhamel-inhomogeneous} is given by:
\begin{equation}
    \label{eq::system-approximation}
    v = \schrodinger_0 u_0 + \mathcal{K} W_\Lambda v + g,
\end{equation}
The coefficients $\phi$, $\psi$, and $W$ approximate respectively $\mathcal{E}\eta$, $\mathcal{E}(t\eta)$, and $V$.
In \eqref{eq::system-approximation}, we use $v$ instead of $u$ for the unknown to distinguish the difference between solutions.
From Lemma~\ref{lem::duhamel-op-str-est}, the following estimate holds.

\begin{lemma}
    \label{lem::K-est}
    If $b > 1/2$, then $\|\mathcal{K}\|_{Y^{b-1,b}} \lesssim_b \|\phi\|_{H^b_t} + \|\psi\|_{H^b_t}$.
\end{lemma}

The well-posedness of the approximation equation \eqref{eq::system-approximation} follows directly from Proposition~\ref{prop::schrodinger-duhamel-well-posedness} and a perturbative argument.

\begin{proposition}
    Following Proposition~\ref{prop::schrodinger-duhamel-well-posedness}, for all $u_0 \in L^2_x$ and all $g \in X^b$, the equation \eqref{eq::system-approximation} admits a unique solution $v \in X^b$, provided that
    \begin{equation*}
        \delta
        = \max\{ \|V-W\|_{Y^{b,b-1}}, 
        \|\mathcal{E}(\eta) - \phi\|_{H^b_t},
        \|\mathcal{E}(t\eta) - \psi\|_{H^b_t}\}
    \end{equation*}
    is sufficiently small.
    Moreover, this solution $v$ satisfies
    \begin{equation}
        \label{eq::est-energy-approx}
        \|v\|_{X^b} \lesssim_{d,b} \|u_0\|_{L^2_x} + \|g\|_{X^b}.
    \end{equation}
\end{proposition}

\begin{proof}
    Rewrite the approximate equation \eqref{eq::system-approximation} as
    \begin{equation*}
        v = \schrodinger_0 u_0 + \mathcal{E} \eta \mathfrak{D} V_\Lambda v + \mathcal{E} \eta \mathfrak{D} (V_\Lambda - W_\Lambda) v + (\mathcal{K}_{\mathcal{E}(\eta) - \phi,\mathcal{E}(t \eta) - \psi}) W_\Lambda v + g.
    \end{equation*}
    By Proposition~\ref{prop::schrodinger-duhamel-well-posedness} and Lemma~\ref{lem::K-est}, we obtain
    \begin{equation*}
        \|\mathcal{E} \eta \mathfrak{D} (V_\Lambda - W_\Lambda) + (\mathcal{K}_{\mathcal{E}(\eta) - \phi,\mathcal{E}(t \eta) - \psi}) W_\Lambda v\|_{Y^{b,b}} \lesssim \delta (1 + \|V\|_{Y^{b,b-1}}).
    \end{equation*}
    We conclude with the same fixed point argument as Proposition~\ref{prop::schrodinger-duhamel-well-posedness}.
\end{proof}

\begin{remark}
    \label{rem::Gamma-localized-solution}
    If $u_0 \in L^2_\Gamma$ and $g \in X^b_\Gamma$ with $\Gamma$ parallel to $\Lambda$, then solutions to the periodized equation \eqref{eq::schrodinger-duhamel-inhomogeneous} and the approximate equation \eqref{eq::system-approximation} are both in $X^b_\Gamma$.
    To show this, it suffices to restrict the fixed point argument to the subspace $X^b_\Gamma$.
    Therefore, Lemma~\ref{lem::observation-integral-invariance} applies to their solutions.
\end{remark}

\subsection{Cluster Locality}

Fix a primitive sublattice $\Lambda$ and $\Gamma \in \affine_\Lambda$.
Take the cluster decomposition of scale $R\ge 1$.
Then use the simplified notations:
\begin{equation*}
    \Pi_\alpha = \Pi_{N^\alpha_\Gamma}, \quad
    \Pi^c = \Pi_{N^c_\Gamma}.
\end{equation*}

We say that an operator $\mathcal{A} : X^b_\Gamma \to X^{b'}_\Gamma$, where $b,b' \in \mathbb{R}$, is cluster $R$-local over $\Gamma$ if, for any pair of distinct indexes $\alpha, \beta \in \mathscr{I}(\Gamma;R)$, there holds
\begin{equation*}
    \Pi_\alpha \mathcal{A} \Pi_\beta = 0.
\end{equation*}

\begin{lemma}
    \label{lem::cluster-locality-example}
    If $W \in \trig_x$ with $\deg W \le R$, then $W_\Lambda$ is cluster $R$-local over $\Gamma$.
    If $\phi, \psi \in \trig_t$ with $\deg \phi \le R$ and $\deg \psi \le R$, then $\mathcal{K}$ is cluster $R$-local over $\Gamma$.
\end{lemma}

\begin{proof}
    Since the Fourier support of a function expands at most to its $R$-neighbor\-hood after multiplication with trigonometric polynomials of degrees $\le R$, the cluster locality for $W_\Lambda$ follows directly from \eqref{eq::cluster-far-away-N}.

    To verify the cluster locality of $\mathcal{K}$, we show that $\phi \Theta$, $\phi \schrodinger_0 \mathcal{T} \Theta$ and $\psi \schrodinger_0 \mathcal{P}$ have this property.
    Clearly, the cluster locality of $\phi \Theta$ follows from the same reasoning as for $W_\Lambda$.
    To verify the cluster locality for operators $\phi \schrodinger_0 \mathcal{T} \Theta$ and $\psi \schrodinger_0 \mathcal{P}$, one notices that the Fourier supports of $\schrodinger_0 \mathcal{T} \Theta \Pi_\beta f$ and $\schrodinger_0 \mathcal{P} \Pi_\beta f$ are contained in $\mathbb{Z} \times Z^\beta_\Gamma$, so are the Fourier supports of $\phi \schrodinger_0 \mathcal{T} \Theta \Pi_\beta f$ and $\psi \schrodinger_0 \mathcal{P} \Pi_\beta f$ since $\phi$ and $\psi$ only depend on time.
    These Fourier supports are disjoint from $\mathbb{Z} \times Z^\alpha_\Gamma \supset N^\alpha_\Gamma$ since $Z^\alpha_\Gamma \cap Z^\beta_\Gamma = \emptyset$ when $\alpha \ne \beta$.
\end{proof}

The cluster locality may not preserve under composition.
However, the following lemma and Lemma~\ref{lem::est-far-from-paraboloid} shows that it almost holds.

\begin{lemma}
    \label{lem::cluster-locality-composition}
    Let $\mathcal{A} : X^b_\Gamma \to X^{b'}_\Gamma$ and $\mathcal{B} : X^{b'}_\Gamma \to X^{b''}_\Gamma$ (where $b,b',b'' \in \mathbb{R}$) be cluster $R$-local.
    Then for any $\alpha,\beta \in \mathscr{I}(\Gamma;R)$, there holds
    \begin{equation*}
        \Pi_\alpha \mathcal{A} \mathcal{B} \Pi_\beta
        = \Pi_\alpha \mathcal{A} \Pi^c \mathcal{B} \Pi_\beta.
    \end{equation*}
\end{lemma}

\begin{proof}
    Using the decomposition $1 = \sum_{\gamma \in \mathscr{I}} \Pi_\gamma + \Pi^c$, one obtains
    \begin{equation*}
        \Pi_\alpha \mathcal{A} \mathcal{B} \Pi_\beta
        = \sum_{\gamma \in \mathscr{I}} \Pi_\alpha \mathcal{A} \Pi_\gamma \mathcal{B} \Pi_\beta + \Pi_\alpha \mathcal{A} \Pi^c \mathcal{B} \Pi_\beta
        = \Pi_\alpha \mathcal{A} \Pi^c \mathcal{B} \Pi_\beta.
    \end{equation*}
    Indeed, since $\alpha \ne \beta$, for each $\gamma$, either $\gamma \ne \alpha$ or $\gamma \ne \beta$.
    If $\gamma \ne \alpha$, then $\Pi_\alpha \mathcal{A} \Pi_\gamma = 0$; if $\gamma \ne \beta$, then we use $\Pi_\gamma \mathcal{B} \Pi_\beta = 0$.
\end{proof}

\subsection{Cluster Decomposition }

In this section, let $u \in X^b_\Gamma$ is the solution to the periodized Schrödinger equation \eqref{eq::schrodinger-duhamel-inhomogeneous} with initial data $u_0 \in L^2_\Gamma$ and exterior force $g = 0$.

Take the cluster decomposition of $\mathbb{Z} \times \Gamma$ of scale $R \ge 1$.
Recall that $\Gamma^\alpha$ is the smallest primitive affine sublattice containing $Z^\alpha_\Gamma$.
Assume that $\Gamma^\alpha$ is parallel to $\Lambda^\alpha$.
Denote 
\begin{equation*}
    V_\alpha = V_{\Lambda^\alpha},
    \quad
    u_{0,\alpha} = \Pi_{Z^\alpha_\Gamma} u_0.
\end{equation*}
Then let $u_\alpha \in X^b_{\Gamma^\alpha}$ solve the equation
\begin{equation*}
    u_\alpha = \schrodinger_0 u_{0,\alpha} + \mathcal{E} \eta \mathfrak{D} V_\alpha u_\alpha.
\end{equation*}
We aim to prove the following proposition.

\begin{proposition}
    \label{prop::est-cluster-decoupling}
    There holds the estimate
    \begin{equation}
        \label{eq::est-cluster-decoupling}
        \bigl| \|\chi u\|_{L^2_{t,x}}^2 - \|\chi u_*\|_{\ell^2 L^2_{t,x}}^2 \bigr|
        \lesssim_{d,b,\chi,V} o(1)_{R \to \infty} \|u_0\|_{L^2_x}^2.
    \end{equation}
\end{proposition}

For $R \ge 1$, choose $\phi, \psi \in \trig_t$, $W \in \trig_x$, and $\zeta \in \trig_{t,x}$, all of degrees $\le R$, such that the following quantities are all of order $o(1)_{R \to \infty}$:
\begin{equation*}
    \|\phi-\mathcal{E}(\eta)\|_{H^b_t},
    \ \ 
    \|\psi-\mathcal{E}(\eta)\|_{H^b_t},
    \ \ 
    \|V-W\|_{Y^{b,b-1}},
    \ \ 
    \||\chi|^2-|\zeta|^2\|_{Y^{b,-b}}.
\end{equation*}
Let $v \in X^b_\Gamma$ solve \eqref{eq::system-approximation} with $u_0 \in L^2_\Gamma$ and $g = 0$.
Then $w = u - v$ satisfies
\begin{equation*}
    w = \mathcal{E} \eta \mathfrak{D} V_\Lambda w + \mathcal{E} \eta \mathfrak{D} (V_\Lambda - W_\Lambda) v + (\mathcal{K}_{\mathcal{E}(\eta) - \phi,\mathcal{E}(t \eta) - \psi}) W_\Lambda v.
\end{equation*}
By the estimates \eqref{eq::est-energy-periodized} and \eqref{eq::est-energy-approx}, there holds
\begin{equation*}
    \|\chi u - \chi v\|_{L^2_{t,x}}^2
    \le \|\chi\|_{Y^{b,-b}} \|u - v\|_{X^b}^2
    \lesssim_{d,b,\chi, V} o(1)_{R\to \infty} \|u_0\|_{L^2_x}^2.
\end{equation*}
Let $v_\alpha = \Pi_\alpha v$, $v^c = \Pi^c v$ and $\tilde{v} = v - v^c$.
By orthogonality and \eqref{eq::cluster-far-away-N}, one has $\|\zeta \tilde{v}\|_{L^2_{t,x}} = \|\zeta v_*\|_{\ell^2 L^2_{t,x}}$.
Therefore, by the triangular inequality,
\begin{align*}
    \bigl| \|\chi v\|_{L^2_{t,x}}^2 & - \|\chi v_*\|_{\ell^2 L^2_{t,x}}^2 \bigr| \\
    & \le \|\chi v^c\|_{L^2_{t,x}}^2 + \bigl|\|\chi \tilde{v}\|_{L^2_{t,x}}^2 - \|\zeta \tilde{v}\|_{L^2_{t,x}}^2 \bigr| + \bigl| \|\chi v_*\|_{\ell^2 L^2_{t,x}}^2 - \|\zeta v_*\|_{\ell^2 L^2_{t,x}}^2 \bigr| \\
    & \lesssim_\chi \|v^c\|_{X^{b-\epsilon}}^2 + o(1)_{R \to \infty} \bigl(\|\tilde{v}\|_{L^2_{t,x}}^2 + \|v_*\|_{\ell^2 L^2_{t,x}}^2 \bigr)  \\
    & \lesssim_{d,b,\chi, V} o(1)_{R \to \infty} \|v\|_{L^2_{t,x}}^2
    \lesssim_{d,b,\chi, V} o(1)_{R \to \infty} \|v_0\|_{L^2_x}^2.
\end{align*}

To finish the proof of Proposition~\ref{prop::est-cluster-decoupling}, it remains to show that
\begin{equation*}
    \|\chi u_* - \chi v_*\|_{\ell^2 L^2_{t,x}}
    \lesssim_{d,b,\chi, V} o(1)_{R \to \infty} \|u_0\|_{L^2_x}.
\end{equation*}
For this purpose, let us derive the equation satisfied by $w_\alpha = u_\alpha - v_\alpha$.

\begin{lemma}
    For each $\alpha \in \mathscr{I}(\Gamma;R)$, there holds
    \begin{equation*}
        w_\alpha = \mathcal{E} \eta \mathfrak{D} V_\alpha w_\alpha + \mathcal{J}_\alpha v
    \end{equation*}
    where, denoting $W_\alpha = W_{\Lambda^\alpha}$, the operator $\mathcal{J}_\alpha$ is defined by
    \begin{equation*}
        \mathcal{J}_\alpha = (\mathcal{E} \eta \mathfrak{D} V_\alpha - \mathcal{K} W_\alpha) \Pi_\alpha - \bigl([\Pi_\alpha,\mathcal{K}W_\alpha] \Pi_\alpha + \Pi_\alpha [\Pi_\alpha,\mathcal{K} W_\Lambda]\bigr).
    \end{equation*}
\end{lemma}

\begin{proof}
    The lemma would follow immediately after deriving the equation for $v_\alpha$.
    Applying $\Pi_\alpha$ to both sides of \eqref{eq::system-approximation} with $g=0$ and using the identity
    \begin{equation*}
        \Pi_\alpha \mathcal{K} W_\Lambda \Pi_\alpha = \Pi_\alpha \mathcal{K} W_\alpha \Pi_\alpha
    \end{equation*}
    which is not difficult to verify, one obtain that
    \begin{align*}
        v_\alpha - \schrodinger_0 u_{0,\alpha}
        &  = \Pi_\alpha \mathcal{K} W_\Lambda v_\alpha
        = \Pi_\alpha \mathcal{K} W_\Lambda v_\alpha + \Pi_\alpha [\Pi_\alpha,\mathcal{K} W_\Lambda] v \\
        & = \mathcal{K} W_\alpha v_\alpha + [\Pi_\alpha,\mathcal{K}W_\alpha] v_\alpha + \Pi_\alpha [\Pi_\alpha,\mathcal{K} W_\Lambda] v.
        \qedhere
    \end{align*}
\end{proof}

To understand the commutators in the definition of $\mathcal{J}_\alpha$, let us notice that by Lemmas~\ref{lem::cluster-locality-example} and~\ref{lem::cluster-locality-composition}, for $W' \in \{W_\Lambda, W_\alpha\}$, there holds
    \begin{align*}
        &\phantom{{}={}}
        [\Pi_\alpha,\mathcal{K} W']
        = \Pi_\alpha\mathcal{K} W' (1-\Pi_\alpha) - (1-\Pi_\alpha) \mathcal{K} W' \Pi_\alpha \\
        & = \Pi_\alpha\mathcal{K} W' (1-\Pi_\alpha- \Pi^c) - (1-\Pi_\alpha - \Pi^c) \mathcal{K} W' \Pi_\alpha + \Pi_\alpha\mathcal{K} W' \Pi^c - \Pi^c \mathcal{K} W' \Pi_\alpha \\
        & = \Pi_\alpha \mathcal{K} (\Pi^c W' (1- \Pi^c) +  W' \Pi^c )
        - ((1-\Pi^c) \mathcal{K} \Pi^c + \Pi^c \mathcal{K}) W' \Pi_\alpha.
    \end{align*}
From this, by Lemma~\ref{lem::est-far-from-paraboloid} and Lemma~\ref{lem::K-est}, we obtain
\begin{equation*}
    \|\chi w_*\|_{\ell^2 L^2_{t,x}}^2
    \lesssim \|\chi\|_{Y^{b,-b}} \|\mathcal{J}_* v\|_{\ell^2 X^b}^2
    \lesssim_{d,b,\chi,V} o(1)_{R \to \infty} \|u_0\|_{L^2_x}^2.
\end{equation*}

\subsection{Mathematical Induction}

As in the case of the free propagator, we use mathematical induction to derive weak observability estimates.

\begin{lemma}
    If $\chi \in L^2_{t,x}$, $\Lambda_0 = \{0\}$, and $\dim \supp \widehat{u}_0 = 0$, then
    \begin{equation*}
        \|\chi \schrodinger_{\Lambda_0} u_0\|_{L^2(\mathbb{R}_t\times\torus^d_x)} = \|\chi\|_{L^2_{t,x}} \|u_0\|_{L^2_x}.
    \end{equation*}
\end{lemma}

\begin{proof}
    Let $v(t) = e^{-it m} e^{-it\hamiltonian_{\Lambda_0}} u_0$, where $m = (2\pi)^{-d/2} \widehat{V}(0)$.
    Then $v$ satisfies the free Schrödinger equation and we conclude by Lemma~\ref{lem::zero-dim-observability-free}.
\end{proof}

\begin{proposition}
    \label{prop::weak-observability}
    Assume that \eqref{eq::observability-Lambda-V-general} holds for all primitive sublattices of dimensions $\le \nu \le d-1$.
    Let $\Lambda$ be a $(\nu+1)$-dimensional primitive sublattice.
    For every $\Gamma \in \affine_\Lambda$, there exists a finite subset $B_\Gamma \subset \Gamma$ such that, if $u_0 \in L^2_\Gamma$, then
    \begin{equation}
        \label{eq::weak-observability}
        \|u_0\|_{L^2_x}^2
        \lesssim_{d,\chi,V} \|\chi \schrodinger_\Lambda u_0\|_{L^2(\mathbb{R}_t \times \torus^d_x)}^2
        + \|\Pi_{B_\Gamma} u_0\|_{L^2_x}^2.
    \end{equation}
\end{proposition}

\begin{proof}
    It suffices to establish this estimate when $\chi$ satisfies in addition the support condition $\supp \chi \subset \{\eta=1\} \times \torus^d_x$.
    It remains to prove
    \begin{equation*}
        \|u_0\|_{L^2_x}^2
        \lesssim_{d,\chi,V} \|\chi u\|_{L^2_{t,x}}^2
        + \|\Pi_{B_\Gamma} u_0\|_{L^2_x}^2,
    \end{equation*}
    when $u$ solves the homogeneous periodized Schrödinger equation \eqref{eq::schrodinger-duhamel-inhomogeneous} with initial data $u_0$ and vanishing exterior force $g=0$.
    
    Let $R$ be sufficiently large and proceed as the proof of Proposition~\ref{prop::weak-observability-V=0}.
    The weak observability \eqref{eq::weak-observability} follows immediately from \eqref{eq::est-cluster-decoupling}.
    Indeed, if $\alpha \in \mathscr{I}_\sharp(\Gamma;R)$, then $\|u_{0,\alpha}\|_{L^2_x} \lesssim_{d,\chi,V} \|\chi u_\alpha\|_{L^2_{t,x}}$ due to the induction hypothesis.
    Summing this over $\alpha \in \mathscr{I}(\Gamma;R)$ and invoking \eqref{eq::est-cluster-decoupling}, we obtain
    \begin{align*}
        \|\chi u\|_{L^2_{t,x}}^2 
        & \gtrsim_{d,\chi,V} \|\chi u_*\|_{\ell^2 L^2_{t,x}}^2 + o(1)_{R \to \infty} \|u_0\|_{L^2_x}^2 \\
        & \gtrsim_{d,\chi,V} \|(1-\Pi_{B_\Gamma}) u_0\|_{L^2_x}^2 + o(1)_{R \to \infty} \|u_0\|_{L^2_x}^2.
        \qedhere
    \end{align*}
\end{proof}

\subsection{Uniqueness-Compactness Argument}

\label{sec::unique-compactness-argument}

Now we derive the observability for all primitive sublattices of dimension $\nu+1$.
First notice that, similarly to \eqref{eq::galilean-free-propagator-relation}, for $u(t) = e^{-it\hamiltonian_\Lambda} u_0$ with $u_0 \in L^2_\Gamma$, if $p \in \Lambda^\perp$, then
\begin{equation*}
    e^{-it\hamiltonian_\Lambda} (e^{ip\cdot x}u_0) = (\galilean_p u)(t,\cdot).
\end{equation*}
Therefore, we may proceed by contradiction as in \S\ref{sec::uniqueness-compactness-V=0}.
This yields a primitive $\Lambda$ of dimension $\nu+1$, $\Gamma \in \affine_\Gamma$, and a sequence of initial data $u_{0,j} \in L^2_\Gamma$ such that 
\begin{equation*}
    \|u_{0,j}\|_{L^2_x} = 1,
    \quad
    \lim_{j \to \infty} \|\chi u_j\|_{L^2(\mathbb{R}_t\times\torus^d_x)} = 0,
\end{equation*}
where $u_j = \schrodinger_\Lambda u_{0,j}$.
We may further assume that $u_{0,j} \rightharpoonup u_0$ weakly in $L^2_x$ and $\mathcal{E}(\eta u_j) \rightharpoonup u$ weakly in $X^b$ where $b > 1/2$ and $\eta \in C_c^\infty(\mathbb{R})$ is chosen in Proposition~\ref{prop::schrodinger-duhamel-well-posedness}.
Clearly, if $|t|$ is sufficiently small, then $u(t) = e^{-it\hamiltonian_\Lambda} u_0$.

When $\supp \chi \subset \{\eta=1\} \times \torus^d_x$, by Proposition~\ref{prop::weak-observability} and the same reasoning as in \S\ref{sec::uniqueness-compactness-V=0}, we deduce that, for some finite subset $B_\Gamma \subset \Gamma$, there holds
\begin{equation*}
    \|\Pi_{B_\Gamma} u_0\|_{L^2}
    = \lim_{j \to \infty} \|\Pi_{B_\Gamma} u_{0,j}\|_{L^2}
    \gtrsim_{d,\chi,V} 1,
\end{equation*}
By Lemma~\ref{lem::semi-continuity}, we have $\|\chi u\|_{L^2_{t,x}} = 0$.
Hence $u$ vanishes almost everywhere in $\supp \chi$.

On the other hand, since $u_0 \ne 0$, we have $ u \ne 0$.
It remains to show that Schrödinger waves vanishing on a space-time domain with positive Lebesgue measure must vanish identically.
First, with an argument similar to the proof of Lemma~\ref{lem::unique-continuation-V=0} we obtain a temporal holomorphic extension of Schrödinger waves.

\begin{lemma}
    \label{lem::schrodinger-holomorphic}
    For $d \ge 1$, $V \in L^q_x$ with $q \in I_d$, $u_0 \in L^2_x$, and almost all $x \in \torus^d$, $u_x : t \mapsto e^{-it\hamiltonian_V} u_0(x)$ extends to a holomorphic function in $H^2(\mathbb{C}^-)$.
\end{lemma}

\begin{proof}
    Let $\{\lambda_n\}_{n\ge 0}$ be eigenvalues of $\hamiltonian_\Lambda$ with domain $D(\hamiltonian_\Lambda) = H^1_x$.
    By Sobolev embedding, if $f$ is an eigenfunction with respect to eigenvalue $-\lambda$, then
    \begin{equation*}
        \lambda \|f\|_{L^2_x}^2 
        = \langle V f, f\rangle_{L^2_x} + \|\nabla f\|_{L^2_x}^2
        \gtrsim - \|V\|_{L^q_x} \|f\|_{H^1_x}^2,
    \end{equation*}
    from which one induces that, all these eigenvalues, except for a finite number of them, are positive.
    Write $u_0 = \sum_{n \ge 0} f_n$ where $f_n \in L^2_x$ belongs to the eigenspace of $\hamiltonian_V$ with respect to eigenvalue $-\lambda_n$.
    One has
    \begin{equation*}
        \int_{\torus^d} \sum_{n \ge 0} |f_n(x)|^2 \diff x
        = \sum_{n \ge 0} \int_{\torus^d} |f_n(x)|^2 \diff x
        = \|u\|_{L^2_x}^2 < \infty.
    \end{equation*}
    Therefore $(f_n(x))_{n \ge 0} \in \ell^2$ for almost all $x \in \torus^d$.
    For those $x$, the following series extends to a holomorphic function in the Hardy space $H^2(\mathbb{C}^-)$:
    \begin{equation*}
        u_x(t) = \sum_{n \ge 0} f_n(x) e^{-it\lambda_n}.
        \qedhere
    \end{equation*}
\end{proof}

As in the proof of Lemma~\ref{lem::unique-continuation-V=0} using the uniqueness theorem of Lusin and Privalov \cite{LusinPrivalov1925}, we deduce that, if $u$ given by $u(t) = e^{-it\hamiltonian_\Lambda} u_0$ for some $u_0 \in L^2_\Gamma$ vanishes on a space-time domain of positive measure, then there exists $\omega \subset \torus^d$ with $|\omega| > 0$ such that $u$ vanishes almost everywhere on $\mathbb{R} \times \omega$.
Consequently, let $\mathscr{N}$ be the space of all $v_0 \in L^2$ such that $e^{-it\hamiltonian_V} v_0(x) = 0$ for almost all $(t,x) \in \mathbb{R} \times \omega$, then
\begin{equation*}
    \dim \mathscr{N} > 0.
\end{equation*}

\begin{lemma}
    If $\dim \mathscr{N} > 0$, then there exists an eigenfunction of $\hamiltonian_\Lambda$ in $\mathscr{N}$ which vanishes almost everywhere on $\omega$.
\end{lemma}

\begin{proof}
    By the weak observability \eqref{eq::weak-observability}, we have then, for all $u_0 \in \mathscr{N}$,
    \begin{equation*}
        \|\Pi_{B_\Gamma} u_0\|_{L^2_x} \gtrsim_{d,\chi,V} \|u_0\|_{L^2_x},
    \end{equation*}
    where $B_\Gamma$ is a finite subset of $\Gamma$.
    Since $\Pi_{B_\Gamma}$ is a compact operator on $L^2_x$, the Riesz lemma implies that $\dim \mathscr{N}_\Gamma < \infty$.
    Now that $\mathscr{N}_\Gamma$ is invariant by $e^{-it\hamiltonian_\Lambda}$ for all $t \in \mathbb{R}$, if $u_0 \in \mathscr{N}_\Gamma$ and $\epsilon > 0$, then one has
    \begin{equation*}
        f_\epsilon 
        \coloneqq \frac{(e^{-i\epsilon \hamiltonian_\Lambda} - 1) u_0}{-i \epsilon}
        = \frac{1}{\epsilon} \int_0^\epsilon \hamiltonian_\Lambda e^{-it \hamiltonian_\Lambda} u_0 \diff t \in \mathscr{N}_\Gamma.
    \end{equation*}
    Define the norm $\|f\|_{*} = \| (1+\hamiltonian_\Lambda)^{-1} f\|_{L^2_x}$.
    Then
    \begin{align*}
        \|f_\epsilon - \hamiltonian_\Lambda u_0\|_{*}
        & \le \frac{1}{\epsilon} \int_0^\epsilon \|(1+\hamiltonian_\Lambda)^{-1} \mathcal{H}_\Lambda (e^{-it \hamiltonian_\Lambda}-1) u_0\|_{L^2_x} \diff t \\
        & \le \|(e^{-it \hamiltonian_\Lambda}-1) u_0\|_{L^\infty_t((0,\epsilon);L^2_x)}
        = o(1)_{\epsilon \to 0},
    \end{align*}
    due to the time continuity of the propagator $e^{-it\hamiltonian_\Lambda}$.
    Therefore $(f_\epsilon)_{0<\epsilon<1}$ is a Cauchy sequence  with respect to $\|\cdot\|_{*}$ and its limit is equal to $\hamiltonian_\Lambda u_0$.
    This implies that, if $u_0 \in \mathscr{N}_\Gamma$, then $\hamiltonian_\Lambda u_0 \in \mathscr{N}_\Gamma$.
    Hence $\mathcal{H}_\Lambda$ restricts to a $\mathbb{C}$-linear map on $\mathscr{N}_\Gamma$.
    Since $\mathscr{N}_\Gamma$ is finite dimensional, $\hamiltonian_\Lambda$ must admit an eigenvector in $\mathscr{N}_\Gamma$.
\end{proof}

To finish the proof of Theorem~\ref{thm::main-general}, it remains to establish a unique continuation result which shows that such eigenfunctions do not exist.
This is Theorem~\ref{thm::unique-continuation}.

\appendix

\section{A Unique Continuation Result}

\label{sec::unique-continuation}

We survey and study the unique continuation property of the Schrödinger operator $-\Delta + V$ on $\torus^d$ where $V \in L^q_x$ with $q \in I_d$ (although the analysis works for general elliptic operators on general manifolds, see \cite{FigueiredoGossez1992unique}).
We consider the equation
\begin{equation}
    \label{eq::elliptic}
    (-\Delta + V) u = 0.
\end{equation}

Our purpose is to prove the following theorem.
\begin{theorem}
    \label{thm::unique-continuation}
    If a weak solution $u \in H^1_x$ to the equation \eqref{eq::elliptic} vanishes on a set of positive Lebesgue measure, then $u=0$.
\end{theorem}

When $d \ge 3$, this result was established by De Figueiredo and Gossez \cite{FigueiredoGossez1992unique} via the following observation: 
if $x_0$ is a Lebesgue point of the null set of $u$, the $u$ vanishes of infinite order at $x_0$; precisely, letting  $B_r$ be the ball of radius $r$ centered at $x_0$, then for all $N > 0$, there holds
\begin{equation*}
    \|u\|_{L^2_x(B_r)} \lesssim_{d,V,N} O(r^N)_{r \to 0}.
\end{equation*}

Theorem~\ref{thm::unique-continuation} then follows from a classic unique continuation result (Theorem~\ref{thm::unique-continuation-point} below) which dates back to Carleman \cite{Carleman1939unicite} and Aronszajn \cite{Aronszajn1957unique} for bounded potentials and was extended by Jerison and Kenig \cite{JerisonKenig1985unique} and Koch and Tataru \cite{KochTataru2001} to more general potentials.
See e.g., \cite{KochTataru2001} for a historical review.

\begin{theorem}
\label{thm::unique-continuation-point}
    If a weak solution $u \in H^1_x$ to the equation \eqref{eq::elliptic} vanishes of infinite order at some $x_0 \in \torus^d$, then $u=0$.
\end{theorem}

To prove Theorem~\ref{thm::unique-continuation} for $d=1,2$, it suffices to show the above observation still holds.
It suffices to modify the proof in \cite{FigueiredoGossez1992unique} by replacing a Sobolev embedding estimate with a Gagliardo--Nirenberg estimate.

\begin{proof}[Proof of Theorem~\ref{thm::unique-continuation}]
    By Sobolev embedding and a classic elliptic estimate, the following estimate (which is essentially \cite[Lem.~1]{FigueiredoGossez1992unique} with $B_{2r}$ replaced with $B_{\sqrt{2}r}$ for later convenience) holds regardless of the dimension:
    \begin{equation}
        \label{eq::doubling-est-poincare}
        r \|\nabla u\|_{B_r} \lesssim_{d,V} \|u\|_{B_{\sqrt{2}r}}.
    \end{equation}
    
    Let $\varphi \in C_c^\infty(\torus^d)$ be such that $\bm{1}_{|x| \le 1} \le \varphi(x) \le \bm{1}_{|x| \le \sqrt{2}}$, and set $\varphi_r(x) = \varphi(\frac{x-x_0}{r})$.
    Let $p \in (2,p_0]$, where $p_0 = \infty$ for $d=1,2$ and $p_0 \in (2,\frac{2d}{d-2}]$ for $d \ge 3$ and let $\alpha = d(\frac{1}{2} - \frac{1}{p}) \in (0,1]$.
    By the Gagliardo--Nirenberg inequality and \eqref{eq::doubling-est-poincare},
    \begin{align*}
        \|u\|_{L^p_x(B_r)}
        & \le \|\varphi u\|_{L^p_x}
        \lesssim_{d,p} \|\nabla(\varphi_r u)\|_{L^2_x}^\alpha \|\varphi u\|_{L^2_x}^{1-\alpha} \\
        & \lesssim_{d,p,\varphi} \Bigl( \|\nabla u\|_{L^2_x(B_{\sqrt{2}r})} + \frac{1}{r} \|u\|_{L^2_x(B_{\sqrt{2}r})} \Bigr)^\alpha \|u\|_{L^2_x(B_{\sqrt{2}r})}^{1-\alpha} \\
        & \lesssim_{d,p,\varphi} \frac{1}{r^\alpha} \|u\|_{L^2_x(B_{2r})}^\alpha \|u\|_{L^2_x(B_{2r})}^{1-\alpha}.
    \end{align*}

    Since $x_0$ is a Lebesgue point of $u^{-1}(0)$, for all $\epsilon > 0$, there exists $E_\epsilon \subset \torus^d$ and $r_\epsilon > 0$ such that $|B_r \setminus E_\epsilon| < \epsilon |B_r|$ for all $r \in (0,r_\epsilon)$.
    For such $r$, we apply Hölder's inequality and the above estimate.
    Then we obtain that
    \begin{align*}
        \|u\|_{L^2_x(B_r)}
        & = \|u\|_{L^2_x(B_r \setminus E_\epsilon)}
        \lesssim \|u\|_{L^r_x(B_r \setminus E_\epsilon)} |B_r \setminus E_\epsilon|^{1/2-1/r} \\
        & = \|u\|_{L^r_x(B_r)} |B_r \setminus E_\epsilon|^{\alpha / d}
        \lesssim_{d,r} \frac{1}{r^\alpha} \|u\|_{L^2_x(B_{2r})}^\alpha \|u\|_{L^2_x(B_{2r})}^{1-\alpha} (\epsilon r^d)^{\alpha / d} \\
        & = \epsilon^{\alpha / d} \|u\|_{L^2_x(B_{2r})}.
    \end{align*}
    To conclude, we proceed verbatim as in \cite{FigueiredoGossez1992unique}.
\end{proof}

\printbibliography

\end{document}